\documentclass[a4paper,12pt]{article}

\usepackage{hyperref}
\usepackage{amsmath}
\usepackage{amsthm}
\usepackage{amsfonts}
\usepackage{amssymb}
\usepackage{subeqnarray}
\usepackage{mathrsfs}
\usepackage{latexsym}
\usepackage{dsfont}
\usepackage{enumerate}
\usepackage[margin=1cm]{geometry}

\newcommand{\ud}{\mathrm{d}}
\newcommand{\beq}{\begin{equation}}
\newcommand{\eeq}{\end{equation}}
\newcommand{\ba}{\begin{eqnarray}}
\newcommand{\ea}{\end{eqnarray}}
\newcommand{\bs}{\boldsymbol}

\newtheorem{theorem}{Theorem}[section]
\newtheorem{remark}{Remark}[section]

\newtheorem{proposition}{Proposition}[section]
\newtheorem{lemma}{Lemma}[section]
\newtheorem{corollary}{Corollary}[section]

\title{Well posedness of a linearized fractional derivative fluid model}

\author{Arnaud Heibig and Liviu Iulian Palade
\thanks{Corresponding author.  E-mail: liviu-iulian.palade@insa-lyon.fr, Fax: +33 472438529}   }

\numberwithin{equation}{section}
\begin{document}

\maketitle

\begin{flushleft}

Universit\'e de Lyon, CNRS, INSA-Lyon \\ Institut Camille Jordan UMR5208 \& P\^ole de Math\'ematiques, B\^at. L\'eonard de Vinci no. 401, 21 avenue Jean Capelle, F-69621, Villeurbanne, France.

\end{flushleft}

\begin{abstract}

The one-dimensional fractional derivative Maxwell model (e.g. Palade et al. Rheol. Acta {\bf 35}, 265, 1996), of importance in modeling the linear viscoelastic response in the glass  transition region, has been generalized in Palade et al. Int. J. Non-Linear Mech. {\bf 37}, 315, 1999, to objective three-dimensional  constitutive equations (CEs) for both fluids and solids.  Regarding the rest state stability of the fluid CE, in Heibig and Palade J. Math. Phys. {\bf 49}, 043101, 2008, we gave a proof for the existence of weak solutions to the corresponding boundary value problem.  The aim of  this work is to achieve the study of the existence  and uniqueness of the aforementioned solutions and to present smoothness results.

\end{abstract} 

\begin{flushleft}


\end{flushleft}

\section{Introduction}\label{intro}
Fractional derivative constitutive equations (CEs) have been found to accurately predict, among others, the stress relaxation of viscoelastic fluids in the glass transition and glassy (high frequency) states.  The experimental behavior of storage $G'$ and loss $G''$ moduli (obtained upon using the time - temperature superposition principle - see \cite{pal2}, \cite{pal3}) of a linear, narrow molecular weight series of polybutadienes is exceptionally well predicted by linearized fractional derivative models as can be reckoned from   \cite{pal1}.  Polybutadienes are of utter importance for the tire industry, for manufacturing certain solid propergols, etc. Similar excellent agreements between frequency sweep experimental data obtained on other polymers (e.g. polystyrenes) and theoretical predictions of linear fractional derivative models are reported in \cite{fr,hey1, main}.

The object of study is the below given objective, fractional derivative viscoelastic (incompressible) fluid constitutive equation (CE) (see \cite{pal4})

\begin{eqnarray}\label{ce}
\lefteqn{ \mathbf{S}(t)+\lambda^{\alpha}\mathbf{F}(t)\bigg\{\int_{-\infty}^{t}\mu_{1}(t-\tau)
\mathbf{F}^{-1}(\tau)\stackrel{\bigtriangledown}{\mathbf{S}}(\tau)\left[\mathbf{F}^{-1}(\tau)\right]^{T}
\ud \tau \bigg\}
\mathbf{F}(t)^{T} } \nonumber\\
& & =G\lambda^{\beta}\mathbf{F}(t)\bigg\{\int_{-\infty}^{t}\mu_{2}(t-\tau)
\mathbf{F}^{-1}(\tau)\mathbf{A}_{1}(\tau)\left[\mathbf{F}^{-1}(\tau)\right]^{T}\ud \tau\bigg\}
\mathbf{F}(t)^{T} 
\end{eqnarray}

Function $\mathbf{S}$ is the (objective) stress tensor and $\stackrel{\bigtriangledown}{\mathbf{S}}$ its objective upper convected derivative defined by (with $D/Dt$ denoting the material derivative and $\mathbf{L}$ the velocity gradient; see for example \cite{hp2},\cite{mor}, \cite{tig}):

\beq \label{ody}
\stackrel{\bigtriangledown}{\mathbf{S}}=\frac{D {\bf S}}{D
t}-\mathbf{L}\mathbf{S}-\mathbf{S}\mathbf{L}^{T},
\eeq 

Function $\mathbf{F}$ is the strain gradient and  $\mathbf{A}_{1}=\nabla\mathbf{u}+(\nabla\mathbf{u})^T={\bf L}+{\bf L}^T$ is the first Rivlin-Ericksen tensor.  The model parameters are such that  $0<\lambda$, $0<\alpha<\beta<1$.  $\mu_{1,2}(t)$ are two memory kernels given by:

\beq\label{mu}
\mu_{1}(t-\tau)=\frac{(t-\tau)^{-\alpha}}{\Gamma(1-\alpha)},\,\,
\mu_{2}(t-\tau)=\frac{(t-\tau)^{-\beta}}{\Gamma(1-\beta)}
\eeq 

The stability of the rest state is now investigated using the linearized theory.  As shown in \cite{pal4} and \cite{pal7}, it is first assumed that the stress tensor $\mathbf{S}=\mathit{O}(\epsilon)$ and  the deformation 
gradient  $\mathbf{F}(t)=\mathbf{1}+\epsilon\mathbf{J}(t)+\mathit{O}(\epsilon^{2})$. Since $\mathbf{L}=\dot{\mathbf{F}}\mathbf{F}^{-1}$ (see for ex. \cite{hp2}, \cite{ddj}, \cite{mor}), $\mathbf{L}=\mathit{O}(\epsilon)$.  Hence the velocity $\mathbf{u}=\mathit{O}(\epsilon)$, and the first Rivlin-Ericksen tensor $\mathbf{A}_{1}=\mathit{O}(\epsilon)$ as well.  Therefore, keeping only terms
of $\mathit{O}(\epsilon)$, within the linear response theory eq.\eqref{ce} reduces to :

\beq\label{law2}
\mathbf{S}(t)+\lambda^{\alpha}\int_{-\infty}^{t}\mu_{1}(t-\tau)\frac{\partial \mathbf{S}(\tau)}{\partial
\tau}\ud\tau=G\lambda^{\beta}\int_{-\infty}^{t}\mu_{2}(t-\tau)\mathbf{A}_{1}(\tau)\ud \tau
\eeq

Next, assume the fluid is contained in a bounded volume $\Omega\subset\mathbb{R}^3$ whose boundary $\partial\Omega$ is sufficiently smooth, and set in motion at $t=0$.  The CE in eq.\eqref{law2} then takes the form:

\beq\label{law3}
\mathbf{S}(t)+\lambda^{\alpha}\int_{0}^{t}\mu_{1}(t-\tau)\frac{\partial \mathbf{S}(\tau)}{\partial
\tau}\ud\tau=G\lambda^{\beta}\int_{0}^{t}\mu_{2}(t-\tau)\mathbf{A}_{1}(\tau)\ud \tau
\eeq 

The above may be re-written in condensed form using the  Caputo operators $D^{\alpha}_{t}$ and $I^{1-\beta}_{t}$ as:

\beq\label{law4}
\mathbf{S}(t)+\lambda^{\alpha}\, D^{\alpha}_{t}\mathbf{S}=G\lambda^{\beta}\, I^{1-\beta}_{t}\mathbf{A}_{1}
\eeq

where for an absolutely continuous function $f:\mathbb{R}_+\to \mathbb{C}$:

\begin{equation}\label{cfd}
D^{\alpha}_{t}f(t)=\frac{1}{\Gamma(1-\alpha)}
\int^{t}_{0}\frac{f'(\tau)}{(t-\tau)^\alpha} \ud\tau  
\end{equation}

and for $f\in L^1_{\text{loc}}(\mathbb{R}_+)$,

\begin{equation}\label{cfi}
I^{1-\beta}_{t}f(t)=\frac{1}{\Gamma(1-\beta)}\int^{t}_{0}
\frac{f(\tau)}{(t-\tau)^\beta} \ud\tau 
\end{equation}

As shown in \cite{pal4,pal7}, investigating the stability of the rest state is tantamount to studying  the existence and the uniqueness of solutions to the following initial boundary value problem (IBVP):

\begin{subeqnarray}\label{bvp1}
& & \frac{\partial\mathbf{u}}{\partial t}=-\nabla p+\nabla\cdot\mathbf{S}  \label{bvp1 1}\\
& & \mathbf{S}+\lambda^\alpha D^{\alpha}_{t}\mathbf{S}=G\lambda^\beta  I^{1-\beta}_{t}\mathbf{A}_{1},\quad \mathbf{A}_{1}=\nabla\mathbf{u}+(\nabla\mathbf{u})^T  \label{bvp1 2}\\
& & \mathbf{\nabla}\cdot\mathbf{u}=0,\quad\text{in}\quad \lbrack 0,+\infty \lbrack  \times
 \Omega,\, \Omega\subset\mathbb{R}^3 \label{bvp1 3}\\
& & \mathbf{u}=\mathbf{0},\quad\text{in} \quad   \lbrack 0,
+\infty \lbrack \times \partial\Omega \label{bvp1 4}\\
& & \mathbf{u}(t=0)=\mathbf{u}_0,\quad \mathbf{S}(t=0)=\mathbf{S}_0 \label{bvp1 5} \\
\end{subeqnarray}

In the above system of equations we assume ${\bf u}: [0,+\infty[ \times  \Omega \to\mathbb{R}^3$, $\nabla\cdot\mathbf{u}_0=0$,

$p:[0,+\infty[ \times \Omega   \to\mathbb{R}$, $\mathbf{S}: [0,+\infty[ \times \Omega \to \mathscr{M}_{3,3}(\mathbb{R})$, $0<\alpha<\beta<1$.  Denote $\delta=\beta-\alpha>0$.  

A change of variables on $({\bf x},t)$ can be performed to eliminate the CE  parameters $\lambda$ and $G$ (see \cite{pal7}).  This is carried out only for convenience; in no way the generality of this paper results is shrinked down.  Therefore, from now on assume $\lambda=G=1$.  

At this stage recall that an existence result for the initial boundary value problem given in equations \eqref{bvp1}was presented in \cite{pal7}.  The present paper, which is a continuation of \cite{pal7}, is organized as follows: 

\begin{itemize}
 \item Section \ref{wfbvp} presents the weak formulation of the boundary value problem.
\item Section \ref{wf} is devoted to proving the existence and uniqueness of the solutions.  We further on use the existence theorem obtained in \cite{pal7} to state a general existence and uniqueness result.  
\item Section \ref{so} deals with the functional framework within which the solution continuity at $t=0$ is proved.
\item Section \ref{n} presents the proof of the solution continuity at $t=0$.  
\item Section \ref{ie} contains results on the solution smoothness.
\end{itemize}

\section{Weak formulation of the IBVP}\label{wfbvp}

All time-depending functions involved in the current stability analysis, save for when stated otherwise, are causal functions (i.e. set equal to zero on $\mathbb{R}_-$).  Hence the convolution in time is simply $(f\ast g)(t):=\displaystyle \int_0^t f(s)g(t-s)\ud s$.

We first present the weak formulation of the boundary value problem eqs.\eqref{bvp1}: find ${\bf u}\in\mathscr{C}^0([0,+\infty[,\,L^2(\Omega)^3)   $ $\cap L^1_{\text{loc}}(\mathbb{R}_+,H^1_0(\Omega)^3)$, $\nabla\cdot{\bf u}=0$, $[{\bf S}]_{ij}\in\mathscr{C}^0 \left( ]0,+\infty[,\,L^2(\Omega)\cap L^1_{\text{loc}}\left(\mathbb{R}_+ ,L^2(\Omega) \right)  \right)$, $i,j=1,2,3$, such that for any test-functions $\forall\boldsymbol{ \theta}\in\left(H^1_0(\Omega)\right)^3,\, \nabla\cdot\boldsymbol{ \theta}=0,\,\forall {\bf a}\in\left(\mathscr{D}(\Omega)\right)^3,\,\forall 
\psi\in\mathscr{C}^{\infty}_{00}([0,+\infty[)$, where $\mathscr{C}^{\infty}_{00}([0,+\infty[) $ denotes the space of $\mathscr{C}^{\infty} $ class functions that vanish in a neighborhood of $+\infty$, the following equations hold true:

\beq\label{fv1}
\psi(0)\int_{\Omega}{\bf u}_0({\bf x})\cdot\bs{\theta}({\bf x}) \ud {\bf x}+\int^{+\infty}_{0}\int_{\Omega}{\bf u}(t,{\bf x})\cdot\bs{\theta}({\bf x})\psi'(t) \ud {\bf x}\ud t =\int^{+\infty}_{0}\int_{\Omega}({\bf S}(t,{\bf x})\colon\nabla\bs{\theta}({\bf x}))\psi(t) \ud {\bf x} \ud t 
\eeq

\begin{eqnarray}\label{fv2}
& & -\int^{+\infty}_{0}\int_{\Omega}\frac{\psi(\tau)}{\Gamma(1-\alpha)\tau^\alpha}[{\bf S}_0]_{ij}({\bf x})[{\bf a}]_j({\bf x})\ud {\bf x} \ud \tau \nonumber\\[0.5cm]
& & -\int^{+\infty}_{0}\int^{+\infty}_{\tau}\int_{\Omega}\frac{\psi'(t)}{\Gamma(1-\alpha)\tau^\alpha}[{\bf S}]_{ij}(t-\tau,{\bf x})[{\bf a}]_j({\bf x})\ud {\bf x}\ud t\ud\tau \nonumber\\[0.5cm]
& & +\int^{+\infty}_{0}\int_{\Omega}[{\bf S}]_{ij}(t,{\bf x})[{\bf a}]_j({\bf x})\psi(t)\ud {\bf x} \ud t= \nonumber\\[0.5cm]
& & -\frac{1}{\Gamma(1-\beta)}\int^{+\infty}_{0}\int^{t}_{0}\int_{\Omega}\frac{\psi(t)}{(t-\tau)^\beta}\left\{\left(\nabla\cdot{\bf a}\right)[{\bf u}]_i+[{\bf u}\cdot\nabla{\bf a}]_i\right\}(\tau,{\bf x})\ud {\bf x}\ud\tau \ud t
\end{eqnarray}

Summation over repeated indices is understood in equations \eqref{fv1} and \eqref{fv2}  above.

We now detail the functional framework.  Let $V=\{{\bf h}\in \displaystyle\left(H^1_0(\Omega)\right)^3\, \text{s.t.}\, \nabla\cdot {\bf h}=0\}$ be the Hilbert space endowed with the inner product:

\beq\label{hils}
\left\langle {\bf f}|{\bf g} \right\rangle_V=\displaystyle\mathop{\sum}_{i,j=1}^{3}\int_{\Omega}\frac{\partial f_i}{\partial x_j}\frac{\overline{\partial g_i}}{\partial x_j}({\bf x})\ud {\bf x}
\eeq

and denote $\|\,\|_{V}$ the corresponding norm.

The closure of $V$ in $(L^2(\Omega))^3 $ is denoted by $H$, the later space being endowed with the inner product:

\beq\label{hsd}
\left\langle {\bf f}|{\bf g} \right\rangle_H=\displaystyle\mathop{\sum}_{i=1}^{3}\int_{\Omega}f_i\overline{g}_i({\bf x})\ud {\bf x}
\eeq

with $\|\,\|_{H}$ being the corresponding norm.  Let  $0<\lambda_1\leq\lambda_2\leq\dots\lambda_n\leq\dots  \displaystyle \mathop{\longrightarrow}_{n\rightarrow +\infty} +\infty $ and ${\bf w}_i\in V,\,i\in\mathbb{N}^\ast$, be the eigenvalues and the corresponding eigenfunctions   of the Stokes operator in $H$, i.e.:

\beq\label{eig1}
\forall \boldsymbol{\phi}\in V,\, \langle {\bf w}_k\,|\, \boldsymbol{\phi}\rangle_V=\lambda_k\langle {\bf w}_k\,|\, \boldsymbol{\phi}\rangle_H,\quad\text{where}\, \|{\bf w}_k\|_H=1
\eeq

\section{The solution existence and uniqueness}\label{wf}

To prove the solution uniqueness, we first eliminate ${\bf S}$ from equations \eqref{fv1} and \eqref{fv2}.

Denote $\mathscr{L}:=\{f\in L^1_{\text{loc}}(\mathbb{R}_+)\,\text{s.t.}\,\exists M>0,\,\text{so that}\, fe^{-Mt}\in L^1(\mathbb{R}_+) \} \subset L^1_{\text{loc}}(\mathbb{R}_+)  $.  Next, let $f\in L^1_{\text{loc},\mathbb{R}_+}(\mathbb{R})$.   For any $a\in \mathbb{R}$ and $\alpha\in ]0,1[$, define $D^\alpha _{t,a}f$ by:

\beq\label{u1}
\left\langle D^\alpha _{t,a}f,\varphi \right\rangle =\dfrac{1}{\Gamma(1-\alpha)}\left[ -a\int_0^{+\infty}\dfrac{\varphi(\tau)}{\tau^\alpha}\ud \tau-\int_0^{+\infty}f(t)\left(\int_0^{+\infty}\dfrac{\varphi'(t+\tau)}{\tau^\alpha}\ud \tau \right)\ud t  \right] 
\eeq

for any test function $ \varphi\in \mathscr{D}(\mathbb{R})$.

Observe that $D^\alpha _{t,a}f\in\mathscr{D}'(\mathbb{R})$.  Moreover, for any $ f\in \mathscr{L}$, one easily sees that $ x\geq M$, $e^{-xt}D^\alpha _{t,a}f\in\mathscr{S}'(\mathbb{R})$.

The hat $\widehat{(\,)}$ notation to be used below stands for the usual Laplace transform.

\begin{proposition}\label{up1}
Let $f\in \mathscr{L}$.  Then, for $\forall s\in \mathbb{C}$, with $\mathrm{Re}(s)$ large enough, one has $\widehat{D^\alpha _{t,0}f}(s)=s^\alpha\hat{f}(s)$.

\end{proposition}

\begin{proof}
Let $M>0$ such that $e^{-Mt} f\in L^1(\mathbb{R}_+)$.  For any $s\in\mathbb{C},\,\mathrm{Re}(s)\geq M$ and $\varphi(t)=e^{-st}$, Eq.\eqref{u1} - still valid for this particular choice of $\varphi(t)=e^{-st}\notin \mathscr{D}(\mathbb{R})$ - gives:

\beq\label{up1 1}
\left\langle D^\alpha _{t,0}f,e^{-st} \right\rangle =-\dfrac{1}{\Gamma(1-\alpha)} \int_0^{+\infty}f(t)\left(\int_0^{+\infty}\dfrac{-se^{s+\tau}}{\tau^\alpha}\ud \tau \right)\ud t  
\eeq

As $\displaystyle\int_0^{+\infty}\dfrac{e^{-s\tau}}{\tau^\alpha}\ud \tau=\dfrac{\Gamma(1-\alpha)}{s^{1-\alpha}}$, for $\mathrm{Re}(s)\geq M>0$ one gets:

\beq\label{up1 2}
\widehat{D^\alpha _{t,0}f}(s)=\left\langle D^\alpha _{t,0}f,e^{-st} \right\rangle =-\dfrac{s}{\Gamma(1-\alpha)} \int_0^{+\infty}e^{-st}f(t)\Gamma(1-\alpha)s^{\alpha-1}\ud t=s^\alpha\hat{f}(s) 
\eeq
\end{proof}

The following classical result (see \cite{cm}) is stated here within our functional framework.  Recall first that (see also equations 14 and 15 in \cite{pal7}):

\beq\label{wf1}
W_0(t)=\dfrac{\sin(\alpha\pi)}{\pi} \displaystyle\int_0^{+\infty}\dfrac{e^{-rt}r^{\alpha-1}}{r^{2\alpha}+2r^{\alpha}\cos(\alpha\pi)+1}\ud r,\, t\geq 0
\eeq

\beq\label{wf2}
E_\alpha(t)=\dfrac{\sin(\alpha\pi)}{\pi}\displaystyle\int_0^{\infty}\dfrac{r^\alpha e^{-rt}}{r^{2\alpha}+2r^\alpha \cos(\alpha\pi)+1}\ud r, \, t> 0
\eeq

\begin{proposition}\label{up2}
Let $F\in \mathscr{L}$.  Then, for any $ a\in \mathbb{R}$ and any $ \alpha\in ]0,1[$, the equation

\beq\label{up2 1} 
D^\alpha_{t,a}f+f=F
\eeq

has a unique solution $f\in \mathscr{L}$, given by $f=E_\alpha\displaystyle \ast F+aW_0$. 

\end{proposition}

\begin{proof}

\textit{Existence}:

Assume $F\in \mathscr{C}^1(\mathbb{R}_+)$.  Then, $f=E_\alpha\displaystyle \ast F+aW_0$ is a solution of Eq.\eqref{up2 1} (cf \cite{cm}).  Now, if one assumes that $F\in \mathscr{L}$, then there exists $  (F_n)_{n\in \mathbb{N}^\ast}$, $F_n\in \mathscr{C}^1(\mathbb{R}_+)$ such that $F_n \displaystyle\mathop{\longrightarrow}_{n\to+\infty} F $ in $L^1_{\text{loc}}(\mathbb{R}_+)$.  Since $E_\alpha\in L^1_{\text{loc}}(\mathbb{R}_+)$, then $f_n=E_\alpha\displaystyle \ast F_n+aW_0 \displaystyle\mathop{\longrightarrow}_{n\to+\infty} E_\alpha\displaystyle\mathop{\ast}_{(t)}F+aW_0 $ in $L^1_{\text{loc}}(\mathbb{R}_+)$.  Hence the equation $D^\alpha_{t,a}f_n+f_n=F_n$, for $n\to+\infty$, becomes $D^\alpha_{t,a}f+f=F$.   

We must next show that $f\in \mathscr{L}$.  Notice first that $\|aW_0\|_{\infty}\leq |a|W_0(t)$.  Therefore $aW_0\in\mathscr{L}$.   Denote $g=e^{-Mt}F$; choose $M>0$ so that $e^{-Mt}F\in L^1(\mathbb{R})$.  Then $\left|E_\alpha\displaystyle\mathop{\ast}_{(t)}F \right|=\left|\displaystyle \int _0^t E_\alpha(t-s)e^{Ms}g(s)\ud s \right|\leq e^{Mt} \left[ E_\alpha\displaystyle\mathop{\ast}_{(t)}g \right]$.  Since $g\in L^1(\mathbb{R}_+)$, and $E_\alpha\in L^1(\mathbb{R}_+)$, then $E_\alpha\displaystyle \ast g\in L^1(\mathbb{R}_+)$.  Finally $E_\alpha\displaystyle \ast F\in \mathscr{L}$ and $f=E_\alpha\displaystyle \ast F+aW_0\in \mathscr{L}$.  

\textit{Uniqueness}:

Let $f,g\in \mathscr{L}$ be two solutions of Eq.\eqref{up2 1}.  Then $D_{t,0}^\alpha (f-g)+(f-g)=0$, from which it follows that $(s^\alpha+1)\widehat{f-g}(s)=0$, for $Re(s)$ large enough.  Therefore $\widehat{f-g}(s)=0$, thus $f=g$. 

\end{proof}

We shall use the following result to prove the uniqueness property:

\begin{lemma}\label{ul1}
For any $g\in\mathscr{L}$, $I_t^{1-\beta}g\in\mathscr{L}$. 
\end{lemma}

\begin{proof}
Since $g\in\mathscr{L}$, there exists $G\in L^1(\mathbb{R}_+)$ and $M>0$ such that  $g=Ge^{-Mt}$.  Therefore, for $t\geq0$ a.e., $\left|I_t^{1-\beta}g(t) \right|\leq K \left|\displaystyle\int_0^t \dfrac{g(t-u)}{u^\beta}\ud u  \right| \leq K \displaystyle\int_0^t |G(t-u)|\dfrac{e^{-M(t-u)}}{u^\beta}\ud u \leq K e^{-Mt}\left(|G|\displaystyle \ast u_\beta e^{-Mt} \right) $.  Now, $G \in L^1(\mathbb{R}_+)$,  $u_\beta e^{-Mt}\in L^1(\mathbb{R}_+)$ leads to $\left(|G|\displaystyle \ast u_\beta e^{-Mt} \right)\in L^1(\mathbb{R}_+)$.  Therefore, $\left|I_t^{1-\beta}g(t) \right|\leq e^{Mt}H(t)$, with $H\in L^1(\mathbb{R}_+)$, which gives $I_t^{1-\beta}g\in\mathscr{L}$.

\end{proof}

Making use of Proposition \ref{up2} and of Lemma \ref{ul1}, we get:

\begin{corollary}[\textbf{Solution uniqueness}]\label{uc1}
Let ${\bf u}_0\in H$ and ${\bf S}_0\in L^2(\Omega)^9$.  The system of equations Eqs.\eqref{fv1}-\eqref{fv2} has at most one solution that belongs to the functional space $ \mathscr{F}:=\{ ({\bf u},{\bf S})  \in \left[ \mathscr{C}^0  (\mathbb{R}_+,H)\cap L^1_{\text{loc}}(\mathbb{R}_+,V)  \right] \times \mathscr{C}^0\left(]0,+\infty[,L^2(\Omega)^9  \right)$,  such that $\|\nabla{\bf u}\|_{L^2(\Omega)^9}\in \mathscr{L},\,\| {\bf S}\|_{L^2(\Omega)^9}\in \mathscr{L}    \} $.  

\end{corollary}

\begin{proof}
Let $({\bf u},{\bf S})\in \mathscr{F}$ be a solution to Eqs.\eqref{fv1}-\eqref{fv2}.  For any test function $\bs{\varphi}\in \mathscr{D}(\Omega)^9$, as a consequence of Eq.\eqref{fv2} and of the fact that ${\bf u}\in L^1_{\text{loc}}(\mathbb{R}_+,V)$, one has

\beq\label{uc1 1}
D^\alpha_{t,\left\langle{\bf S}_0|\bs{\varphi}  \right\rangle_{L^2(\Omega)^9}}\left\langle{\bf S}|\bs{\varphi}  \right\rangle_{L^2(\Omega)^9}+ \left\langle{\bf S}|\bs{\varphi}  \right\rangle_{L^2(\Omega)^9}=I^{1-\beta}_t\left( \left\langle{\bf A}_1|\bs{\varphi}  \right\rangle_{L^2(\Omega)^9} \right) 
\eeq

However, $\left| \left\langle{\bf S}|\bs{\varphi}  \right\rangle_{L^2(\Omega)^9} \right|\leq \|{\bf S}\|_{L^2(\Omega)^9} \|\bs{\varphi}\|_{L^2(\Omega)^9} $. Since $\|{\bf S}\|_{L^2(\Omega)^9}\in \mathscr{L}$ and $\|\nabla{\bf u}\|_{L^2(\Omega)^9}\in \mathscr{L}$,  we infer that $\left\langle{\bf S}|\bs{\varphi}  \right\rangle_{L^2(\Omega)^9}\in \mathscr{L}$, and $\left\langle{\bf A}_1|\bs{\varphi}  \right\rangle_{L^2(\Omega)^9}\in \mathscr{L}$.  Now Lemma \ref{ul1} implies $I_t^{1-\beta}\left\langle {\bf A}_1|\bs{\varphi} \right\rangle_{L^2(\Omega)^9}\in\mathscr{L}$, and  Proposition \ref{up2} leads to 

\begin{eqnarray}\label{uc1 2}
\left\langle {\bf S}|\bs{\varphi} \right\rangle_{L^2(\Omega)^9} & = & E_\alpha \ast  I^{1-\beta}_t\left( \left\langle{\bf A}_1|\bs{\varphi}  \right\rangle_{L^2(\Omega)^9} \right)+\left\langle {\bf S}_0|\bs{\varphi} \right\rangle_{L^2(\Omega)^9}W_0 \nonumber\\
& = & \rho \ast  \left\langle{\bf A}_1|\bs{\varphi}  \right\rangle_{L^2(\Omega)^9}  + \left\langle {\bf S}_0|\bs{\varphi} \right\rangle_{L^2(\Omega)^9}W_0 
\end{eqnarray}

Notice that Eq.\eqref{uc1 2} still holds true for $\bs{\varphi} \in L^2(\Omega)^9$.

Let $ {\bs \theta}\in V$, $\psi\in \mathcal{C}_{00}^{+\infty}([0,+\infty[) $.  
We deduce from \eqref{uc1 2} and \eqref{fv1} that:

\begin{eqnarray}
& & \psi(0)\int_{\Omega} {\bf u}_0({\bf x})\cdot {\bs \theta}({\bf x})\ud{\bf x}+ \int_0^{+\infty}\int_{\Omega} {\bf u}(t,{\bf x}) \cdot {\bs \theta}({\bf x})\psi'(t)\ud{\bf x}\ud t \nonumber \\
& & =-\int_0^{+\infty} \left( \rho \ast  \left\langle{\bf A}_1|\nabla{\bs \theta}  \right\rangle_{L^2(\Omega)^9}  + \left\langle {\bf S}_0|\nabla{\bs \theta} \right\rangle_{L^2(\Omega)^9}W_0 \right)(t)\psi(t) \ud t \label{uc1 3}
\end{eqnarray}

We search for ${\bf u}\in L^1_{\text{loc}}(\mathbb{R}_+,V)$.  In this case, for almost every $t>0$, ${\bf u}$ can be expressed as

\beq\label{wf3}
{\bf u}(t)=\displaystyle \sum_{q=1}^{+\infty}\alpha_q(t){\bf w}_q
\eeq

the series being convergent in $V$.  It follows, by taking ${\bs \theta}={\bf w}_k$ and $\psi\in\mathscr{D}(]0,+\infty[)$ in equation \eqref{uc1 3}, that $\alpha'_k= -\lambda_k \left(\rho \ast \alpha_k \right)-b_k\sqrt{\lambda_k}W_0$, with $b_k:=\displaystyle \int_{\Omega}\left( {\bf S}_0:\nabla{\bf w}_k\right)({\bf x})\ud {\bf x}$, the equality holding true in $ \mathscr{D}'(]0,+\infty[)$.  Recall that $\alpha_k=\left\langle {\bf u}|{\bf w}_k \right\rangle_H \in \mathscr{C}^0(\mathbb{R}_+)$.  As $W_0\in \mathscr{C}^0(\mathbb{R}_+)$, then necessarily $\alpha_k\in  \mathscr{C}^1(\mathbb{R}_+)$.  However, (cf \cite{pal7}) the Cauchy's initial value problem 

\begin{subeqnarray}\slabel{al:a}
\alpha'_k(t) & = & -\lambda_k \left( \rho \ast \alpha_k \right)(t) -  b_k \sqrt{\lambda_k} W_0(t)  \slabel{al:a1}\\
\alpha_k(0) & = & \alpha_k^0 \slabel{al:a2}
\end{subeqnarray}

has a unique solution in $\mathscr{C}^1(\mathbb{R}_+)$.  The uniqueness of the solution ${\bf u}$ is thus proved, and that of ${\bf S}$ follows. 
 
\end{proof}

We now state an existence and uniqueness result.  Denote $\mathscr{C}^0_b(\mathbb{R}_+^\ast,V):=\{ {\bf u} \in \mathscr{C}^0 (\mathbb{R}_+^\ast,V) \, \text{s.t.} \, \displaystyle\mathop{\sup}_{t\geq1}\|{\bf u}(t)\|_V < +\infty \}$.  The functional space $\mathscr{C}^0_b(\mathbb{R}_+^\ast,L^2(\Omega)^9)$ is defined in a similar way.

\begin{theorem}[\textbf{First Existence and Uniqueness Theorem}]\label{ut1}

Let ${\bf u}_0 \in H$, ${\bf S}_0 \in L^2(\Omega)^9$.  Then the boundary value problem given by the system of equations \eqref{fv1}-\eqref{fv2} has a unique solution 

\beq\label{ut11}
({\bf u},{\bf S}) \in \left[ \mathscr{C}^0 (\mathbb{R}_+,H)\cap\mathscr{C}^0_b (\mathbb{R}_+^\ast,V)\cap L^1_{\text{loc}}(\mathbb{R}_+ ,V) \right] \times  \left[ \mathscr{C}^0_b  (\mathbb{R}_+^\ast,L^2(\Omega)^9)\cap L^1_{\text{loc}}(\mathbb{R}_+ ,L^2(\Omega)^9) \right]
\eeq

Moreover, ${\bf u}(0)={\bf u}_0$.  
\end{theorem}

\begin{proof}
The existence of at least one solution

$$({\bf u},{\bf S}) \in \left[ \mathscr{C}^0 (\mathbb{R}_+,H)\cap\mathscr{C}^0  (\mathbb{R}_+^\ast,V)\cap L^1_{\text{loc}}(\mathbb{R}_+ ,V) \right] \times \left[ \mathscr{C}^0  (\mathbb{R}_+^\ast ,L^2(\Omega)^9) 
\cap L^1_{\text{loc}}(\mathbb{R}_+ ,L^2(\Omega)^9) \right] $$ 

follows from Theorem 8.4 in \cite{pal7}.  

It remains to be proved that the solution $({\bf u},{\bf S})$ obtained in \cite{pal7} satisfies

\beq\label{uta4}
\displaystyle\mathop{\sup}_{t\geq1} \|{\bf u}(t)\|_V+\displaystyle\mathop{\sup}_{t\geq1} \|{\bf S}(t)\|_{L^2(\Omega)^9}<+\infty
\eeq

This is essentially contained in the arguments presented in \cite{pal7}.  Indeed, since ${\bf u}\in \mathscr{C}^0 (\mathbb{R}_+,V)$, we write ${\bf u}=\displaystyle\sum_{k=1}^{+\infty}\alpha_k(t){\bf w}_k$, $t>0$.  From equation 125 in  \cite{pal7} one gets

\begin{eqnarray}\label{wf4}
\lambda_k|\alpha_k(t)|^2  & \leq & M \bigg\{  \left[ \left(\dfrac{1}{t^{ \delta/2}}+\dfrac{1}{t^{1- \delta/2}} \right)^2+ 
\dfrac{\left(t\lambda_k^{1/(2-\delta)}\right)^{2(1-\delta/2)}}{t^{2(1-\delta/2)}}e^{-2\gamma t\lambda_k^{1/(2-\delta)}} \right]|\alpha_k^0|^2\nonumber\\
& + &   \left( \dfrac{1}{t^{2(1-\delta)}}+\dfrac{(t\lambda_k^{1/(2-\delta)})^{2(1-\delta)}}{t^{2(1-\delta)}}e^{-2at\lambda_k^{1/(2-\delta)}}\right)|b_k|^2 \bigg\}
\end{eqnarray}

In the above, $\delta=(\alpha-\beta)\in]0,1[$, $a>0$, $\gamma>0$, and $\alpha_k^0 =\langle  {\bf u}_0 | {\bf w}_k \rangle_{H}$.  Clearly ${\bf u}_0\in L^2(\Omega)^3$,  ${\bf S}_0\in L^2(\Omega)^9$ implies that $\displaystyle\sum_{k=1}^{+\infty}|\alpha_k^0|^2<+\infty$ and $\displaystyle\sum_{k=1}^{+\infty}|b_k |^2<+\infty$.  Hence

\beq\label{uta5}
\|{\bf u}\|^2_V=\displaystyle\sum_{k=1}^{+\infty}\lambda_k|\alpha_k(t)|^2 \leq \dfrac{M}{t^{\inf (\delta,2-2\delta)}}\mathop{\longrightarrow}_{t\to+\infty}0
\eeq

and $ {\bf u}_0\in\mathscr{C}_b^0(\mathbb{R}_+^\ast,V)$.  

We use the equation that defines ${\bf S}$ given right below equation 137 in \cite{pal7}.  Then:

\beq\label{uta1}
\|{\bf S}\|_{L^2(\Omega)^9}\leq M \left[ \displaystyle \sum_{k=1}^{+\infty}\lambda_k |\rho\ast\alpha_k(t)|^2+|W_0(t)| \|{\bf S}_0\|_{L^2(\Omega)^9} \right]
\eeq

From \eqref{wf1} we see that 

\beq\label{uta2}
|W_0(t)|\displaystyle \mathop{\longrightarrow}_{t\to+\infty}0
\eeq

Moreover, (see equation 130 in  \cite{pal7})

$$\lambda_k |\rho\ast\alpha_k(t)|^2 \leq M \left[  \left( \dfrac{1}{t^\delta}+\dfrac{\left( t\lambda_k^{1/(2-\delta)}\right)^\delta}{t^\delta}\exp\left[-at\lambda_k^{1/(2-\delta)}\right]  \right)|\alpha_k^0|^2  + \left( \dfrac{1}{t^{2\epsilon}}+\exp(-2at) \right)|b_k|^2 \right],\epsilon>0, a>0$$

Hence

\beq\label{uta3}
\displaystyle\sum_{k=1}^{+\infty}\lambda_k|\alpha_k(t)|^2 \leq \dfrac{M}{t^{\inf (\delta,2\epsilon)}}, \quad \text{for}\quad t\geq 1
\eeq

Finally (see \eqref{uta1}-\eqref{uta3}): $\|{\bf S}\|_{L^2(\Omega)^9} \displaystyle \mathop{\longrightarrow}_{t\to+\infty}0$.

It follows that ${\bf S}\in\mathscr{C}^0_b(\mathbb{R}_+^\ast,L^2(\Omega)^9)$.  The existence of at least one solution belonging to the functional space of \eqref{ut11} is thus proved.

The uniqueness of such a solution results from Corollary \ref{uc1} and from the fact that

$$\left[ \mathscr{C}^0 (\mathbb{R}_+,H)\cap\mathscr{C}^0_b (\mathbb{R}_+^\ast,V)\cap L^1_{\text{loc}}(\mathbb{R}_+,V) \right] \times \left[ \mathscr{C}^0_b (\mathbb{R}_+^\ast,L^2(\Omega)^9) \cap L^1_{\text{loc}}(\mathbb{R}_+ ,L^2(\Omega)^9) \right] \subset \mathscr{F}$$.

\end{proof}

\section{Functional spaces}\label{so} 

In order to prove the continuity of the $({\bf u},{\bf S})$ at $t=0$ we recall several classical functional spaces (see also \cite{lim} and \cite{rt}).

Denote ${\bs \epsilon}_k=\displaystyle \dfrac{\nabla {\bf w}_k}{\sqrt{\lambda_k}}$, $k\in\mathbb{N}^\ast$.  Let $\Pi:L^2(\Omega)^9\to L^2(\Omega)^9$ be the orthogonal projection operator of $L^2(\Omega)^9 $ onto $\left[\text{Vect}({\bs \epsilon}_k)_{k\in\mathbb{N}^\ast} \right]^{\bot}$, $\theta\geq0$.  

For any ${\bf f}\in L^2(\Omega)^9$, denote $\|{\bf f}\|^2_{D_\theta}:=\displaystyle \sum_{q=1}^{+\infty}\lambda_q^\theta |\left\langle {\bf f}|{\bs \epsilon}_q \right\rangle_{L^2(\Omega)^9} |^2+\|\Pi({\bf f})\|^2_{L^2(\Omega)^9}$.  Let $D_\theta :=\{{\bf f} \in L^2(\Omega)^9 \,\text{s.t.}\,\|f\|_{D_\theta}<+\infty \}$.  

For any ${\bf f},{\bf g} \in D_\theta $, denote $\left\langle {\bf f}|{\bf g}\right\rangle_{D_\theta}:=\displaystyle \sum_{q=1}^{+\infty}\lambda_q^\theta \left\langle {\bf f}|{\bs \epsilon}_q\right\rangle_{L^2(\Omega)^9}  \overline{\left\langle {\bf g}|{\bs \epsilon}_q \right\rangle} _{L^2(\Omega)^9}  +\left\langle \Pi({\bf f})|\Pi({\bf g})\right\rangle_{_{L^2(\Omega)^9}} $.  The functional space $(D_\theta, \left\langle\,|\,\right\rangle_{D_\theta})$ is a Hilbert space.  

For any ${\bf f}\in H$, let $\| {\bf f} \|^2_{H_\theta}:=\displaystyle \sum_{q=1}^{+\infty}\lambda_q^\theta |\left\langle {\bf f}|{\bf w}_q \right\rangle_{H}|^2$.  Next, let $H_\theta :=\{{\bf f}\in H,\, \|{\bf f}\|_{H_\theta}<+\infty \}$.  For any $f,g\in H_\theta$, $\left\langle f|g \right\rangle_{H_\theta}:=\displaystyle \sum_{k=1}^{+\infty}\lambda_q^\theta \left\langle f|{\bf w}_k \right\rangle_{H}  \overline{\left\langle g|{\bf w}_k \right\rangle}_{H}$.   As the sequence $({\bf w}_k)_{k\in\mathbb{N}^\ast}$ is complete in $L^2(\Omega)^9 $, the functional space $(H_\theta,\left\langle \quad|\quad\right\rangle_{H_\theta})$ is a Hilbert space.  

Remark that, for any $0\leq \theta \leq \theta' \leq1 \leq \theta'' $, one has:

\beq\label{so1}
H=H_0\hookleftarrow H_{\theta}\hookleftarrow H_{\theta'}\hookleftarrow H_{1}=V \hookleftarrow H_{\theta''}
\eeq

\beq\label{so2}
L^2(\Omega)^9=D_0\hookleftarrow D_{\theta}\hookleftarrow D_{\theta'}\hookleftarrow D_{1} \hookleftarrow D_{\theta''}
\eeq

The above injections are dense; use of them will be often made from now on.

The following $\Delta_\theta$ spaces are closely related to the $D_\theta$ ones. Let $P:L^2(\Omega)^9 \to  L^2(\Omega)^9$ be the orthogonal projection operator from $L^2(\Omega)^9$ onto $\left[\displaystyle \bigcup_{k=1}^{+\infty} \left\{ {\bs \epsilon}_k,{\bs \epsilon}_k^T \right\} \right]^{\perp}$.

Let $\theta\in\mathbb{R}$.  For any element ${\bf f} \in L^2(\Omega)^9$, denote

\beq\label{sow1}
\|{\bf f}\|^2_{\Delta_\theta}=\displaystyle\sum_{q=1}^{+\infty}\lambda_q^\theta \left| \langle {\bf f} | {\bs \epsilon}_q  \rangle_{L^2(\Omega)^9} \right|^2 + \displaystyle\sum_{q=1}^{+\infty}\lambda_q^\theta \left| \langle {\bf f} | {\bs \epsilon}_q^T  \rangle_{L^2(\Omega)^9} \right|^2 + \left\| P({\bf f})  \right\|^2_{L^2(\Omega)^9}
\eeq

For any $\theta\geq0$, $\Delta_\theta :=\{ {\bf f} \in L^2(\Omega)^9 \,\text{s.t.}\, \left\| {\bf f} \right\|_{\Delta_\theta}<+\infty \}$.  The functional space $\Delta_\theta$ endowed with the inner product defined as: $\forall {\bf f} \in \Delta_\theta$, $\forall {\bf g} \in \Delta_\theta$,

\beq\label{sow2}
\langle {\bf f} | {\bf g} \rangle_{\Delta_\theta} = \displaystyle\sum_{q=1}^{+\infty}\lambda_q^\theta   \langle {\bf f} | {\bs \epsilon}_q \rangle_{L^2(\Omega)^9} \overline{ \langle {\bf g} | {\bs \epsilon}_q  \rangle}_{L^2(\Omega)^9} + \displaystyle\sum_{q=1}^{+\infty}\lambda_q^\theta   \langle {\bf f} | {\bs \epsilon}_q^T \rangle_{L^2(\Omega)^9} \overline{ \langle {\bf g} | {\bs \epsilon}_q^T  \rangle}_{L^2(\Omega)^9} + \langle P({\bf f})|P({\bf g})  \rangle_{L^2(\Omega)^9}
\eeq

is a Hilbert space.

For any $\theta<0$, let $\Delta_\theta$ denote the topological dual space of $\Delta_{-\theta}$, i.e. $\Delta_\theta=\left(\Delta_{-\theta}\right)'$.  The space $\Delta_\theta$  is the completion of $(L^2(\Omega)^9,\|\,\|_{\Delta_\theta})$.  Next, note that whenever $0\leq \gamma\leq \tilde{\gamma}$, the following injections $L^2(\Omega)^9=\Delta_0 \hookleftarrow \Delta_\gamma \hookleftarrow \Delta_{\tilde{\gamma}}$ are dense.  It results that $\Delta_{-\tilde{\gamma}} \hookleftarrow \Delta_{-\gamma} = \left(\Delta_\gamma \right)' \hookleftarrow \left( L^2(\Omega)^9 \right)' \simeq L^2(\Omega)^9= \Delta_0 \hookleftarrow \Delta_\gamma \hookleftarrow \Delta_{\tilde{\gamma}}$, invoking the fact that  $L^2(\Omega)^9 $ and  $\left( L^2(\Omega)^9 \right)'$ are isomorphic to each other.

Next, for any ${\bf f}\in H_0^1(\Omega)^9$, one has:

\begin{eqnarray}\label{so3 3}
\|{\bf f}\|_{\Delta_1}^2 & = &  \sum_{q=1}^{+\infty}\lambda_q \left|\langle {\bf f}|{\bs \epsilon}_q \rangle_{L^2(\Omega)^9} \right|^2+\sum_{q=1}^{+\infty}\lambda_q \left|\langle {\bf f}|{\bs \epsilon}_q^T \rangle_{L^2(\Omega)^9} \right|^2+\|P({\bf f})\|^2_{L^2(\Omega)^9}  \nonumber\\
& = & \sum_{q=1}^{+\infty}  \left|\langle \nabla\cdot {\bf f}|{\bf w}_q \rangle_{L^2(\Omega)^3} \right|^2+\left|\langle \nabla\cdot {\bf f}^T|{\bf w}_q \rangle_{L^2(\Omega)^3} \right|^2+\|P({\bf f})\|^2_{L^2(\Omega)^3} \label{so3 2} \nonumber\\
& \leq & K\|{\bf f}\|^2_{H_0^1(\Omega)^9}
\end{eqnarray}

due to the Poincar\'e's inequality.  Consequently $H_0^1(\Omega)^9\hookrightarrow \Delta_1$ and the restriction $r:\Delta'_1\to H^{-1}(\Omega)^9 $, such that  
$T  \stackrel{r}{\mapsto} T|_{H^1_0(\Omega)^9} $ is continuous.

\begin{lemma}\label{sxl4}

Let $\theta\in\mathbb{R}_+$.

\begin{enumerate}[(a)]
\item\label{a}The sequence $({\bs \epsilon}_k)_{k\in \mathbb{N}^\ast}$ is orthogonal in $D_\theta$.  Moreover, $\|{\bs \epsilon}_k\|_{D_\theta}=\lambda_k^{\theta/2}$. 
\item\label{b} The sequence $({\bs \epsilon}^T_k)_{k\in \mathbb{N}^\ast}$ is orthogonal in $D_\theta$.
\item\label{c} The sequence $({\bs \epsilon}_k+{\bs \epsilon}^T_k)_{k\in \mathbb{N}^\ast}$ is orthogonal in $D_\theta$, and $\|{\bs \epsilon}_k+{\bs \epsilon}^T_k\|_{D_\theta}=(1+\lambda_k^{\theta})^{1/2}$.
\item\label{d} The sequence $({\bf w}_k)_{k\in \mathbb{N}^\ast}$ is orthogonal in $H_\theta$, and $\|{\bf w}_k\|_{H_\theta}=\lambda_k^{\theta/2}$.
\item\label{e} Let ${\bf f}\in L^2(\Omega)^9$.  Denote ${\bf f}=\displaystyle \sum_{k=1}^{+\infty}a_k  {\bs \epsilon}_k + \displaystyle \sum_{k=1}^{+\infty}b_k  {\bs \epsilon}^T_k + P({\bf f})$.  Then: 

$\|{\bf f}\|^2_{\Delta_{-\theta}}=\displaystyle \sum_{k=1}^{+\infty} \lambda_k^{-\theta} \left( |a_k|^2+|b_k|^2\right) + \|P({\bf f})\|^2_{L^2(\Omega)^9}$.
\end{enumerate}

\end{lemma}

\begin{proof}
Observe that (cf. eq.\eqref{eig1}) for any $(k,q)\in \mathbb{N}^{\ast 2} $:

\beq\label{sol1}
\left\langle{\bs \epsilon}_k|{\bs \epsilon}_q \right\rangle _{L^2}=\left\langle{\bf w}_k|{\bf w}_q \right\rangle _{L^2}=\delta_{kq}
\eeq

On the other hand, since $\nabla\cdot {\bf w}_k=0$, 

\begin{eqnarray}\label{sol3}
\left\langle{\bs \epsilon}_k|{\bs \epsilon}^{T}_q \right\rangle _{L^2} & = &  \sum_{i,j}\int_\Omega \dfrac{\partial ({\bf w}_k)_i }{\partial x_j}({\bf x})\overline{\dfrac{\partial ({\bf w}_q)_j }{\partial x_i}}({\bf x}) \ud {\bf x} \label{sol2} \nonumber\\
& = & \int_\Omega \left(\sum_{i}\dfrac{\partial ({\bf w}_k)_i }{\partial x_j}({\bf x}) \right)\left(\overline{ \sum_{j}\dfrac{\partial ({\bf w}_q)_j }{\partial x_i}}({\bf x})\right) \ud {\bf x}=0  
\end{eqnarray}

Hence:

\beq\label{sol4}
\left\langle{\bs \epsilon}_k|{\bs \epsilon}^{T}_q \right\rangle _{L^2} =0
\eeq

The statements (a) to (e) result from Eqs.\eqref{sol1}-\eqref{sol4}.  

\end{proof}

Except for the injection $H_{2k} \hookrightarrow H^{2k}(\Omega)^3$ (see below), the following description of the spaces will not be used in this paper.

Let first $\theta\in[0,1]$.   

Let $\Lambda:V \to H$, such that $\Lambda\displaystyle\left(\sum_{k\in\mathbb{N}^\ast}a_k{\bf w}_k \right) = \displaystyle\sum_{k\in\mathbb{N}^\ast}a_k\sqrt{\lambda_k}{\bf w}_k $.  Then, for any $ ({\bf u},{\bf v})\in V^2$, $\left\langle\Lambda{\bf u}|\Lambda{\bf v} \right\rangle_H =\left\langle{\bf u}|{\bf v} \right\rangle_V$, and $H_\theta=D(\Lambda^\theta)=[V,H]_{1-\theta}$, where $[V,H]_{1-\theta}$ stands for the holomorphic interpolation of spaces $V$ and $H$, and $D(\Lambda^\theta)$ for the domain of $\Lambda^\theta$.  Denote $H^0_0(\Omega)\equiv L^2(\Omega)$.  Let the canonical injection $H\stackrel{i}{\hookrightarrow} H^0_0(\Omega)^3$ and  $V\stackrel{i|_V}{\hookrightarrow} H^1_0(\Omega)^3$ be its restriction.

Then $\displaystyle H_\theta =\displaystyle [V,H]_{1-\theta}\stackrel{i|_{[V,H]_{1-\theta}}}{\hookrightarrow}[H^1_0(\Omega)^3,H^0_0(\Omega)^3]_{1-\theta}\hookrightarrow H^\theta_0(\Omega)^3$ (the last continuous injection $\hookrightarrow$ boils down to an equality $=$ whenever $\theta\neq 1/2$).  Let now $\gamma_n$ denote the normal-trace application.  It is well known that $H=\{ {\bf u}\in H^0(\Omega)^3,\,\text{s.t.}\,\nabla\cdot{\bf u} =0, \,\gamma_n({\bf u})=0 \}$.  From the preceding arguments it results that we have the continuous injection $H_\theta \hookrightarrow H\cap H^\theta_0(\Omega)^3=\{ {\bf u}\in H^\theta_0(\Omega)^3,\,\text{s.t.}\,\nabla\cdot{\bf u} =0, \,\gamma_n({\bf u})=0 \}$, with the space $H\cap H^\theta_0(\Omega)^3$ being endowed with the $H^\theta_0(\Omega)^3$ topology.

Let now $\theta\in\mathbb{N}^\ast$.

As quoted on page 106 in \cite{rt}, $H_2=D(\Lambda^2)=H^2(\Omega)^3\cap V$.  Also, invoking  Agmon - Douglis - Nirenberg's Theorem as stated on page 832 in \cite{dl}, leads to $H_{2k}=D(\Lambda^{2k})\hookrightarrow H^{2k}(\Omega)^3\cap V$, $k\in\mathbb{N}^\ast$.  Here $H^{2k}(\Omega)$ are classical Sobolev spaces.

\section{The continuity of the solution $({\bf u},{\bf S})$ at $t=0$}\label{n}

From now on $({\bf u},{\bf S})$ denotes the solution to equations \eqref{fv1}-\eqref{fv2}, with initial data $({\bf u}_0,{\bf S}_0)\in H\times L^2(\Omega)^9$ (see Theorem \ref{ut1}).  In order to prove continuity results we recall several representation formulas for ${\bf u}$ and ${\bf S}$.    First, functions $\alpha_k$, $k\in \mathbb{N}^\ast$, are defined by equations \eqref{al:a1}-\eqref{al:a2}.  Equivalently, for $x\in\mathbb{R}_+$ (see \cite{pal7}),

\begin{equation} \label{per3 2}
\alpha_k(t) =  \dfrac{1}{2\pi}\left[\lim_{A\rightarrow+\infty}  \int_{-A}^{+A} T_{\lambda_k} (x+iy)e^{(x+iy)t}\ud y \alpha_k^0 
- \lim_{A\rightarrow+\infty}  \int_{-A}^{+A} 
(T_{\lambda_k}w) (x+iy)e^{(x+iy)t}\ud y \sqrt{\lambda_k} b_k \right]
\end{equation}

with $\displaystyle T_\mu (s)=\dfrac{s^{1-\beta}(s^\alpha+1)}{s^{2-\beta}(s^\alpha+1)+\mu}$, $\displaystyle w (s)=\dfrac{1}{s^{1-\alpha}(s^\alpha+1) }$, $\mu\in\mathbb{R}_+$, $s\in \mathbb{C}-\mathbb{R}_-$, and $\alpha_k^0=\langle {\bf u}_0|{\bf w}_k \rangle_H$, $b_k=\displaystyle\int_\Omega ({\bf S}_0:\nabla {\bf w}_k)({\bf x})\ud{\bf x} $.

Notice that eq.\eqref{per3 2} is given in \cite{pal7} only for $x\geq M$.  The general result ($x\in \mathbb{R}_+$)  follows from a simple use of the Cauchy formula; details are omitted.  Regarding function ${\bf S}$, recall the following formula from \cite{pal7}: 

\beq\label{t14 3}
{\bf S}=\sum_{k=1}^{+\infty}(\rho\ast\alpha_k)\otimes(\nabla{\bf w}_k+\nabla^T{\bf w}_k)+W_0\otimes {\bf S}_0
\eeq

Notation $h=f\otimes g$ means $h(x,y)=f(x)g(y)$.  As quoted in \cite{pal7}, the series in \eqref{t14 3} converges in $\mathscr{C}^0(\mathbb{R}_+^\ast,L^2(\Omega)^9)$ and in $L^2_{\text{loc}}(\mathbb{R}_+ ,L^2(\Omega)^9)$.  Here $\rho=E_\alpha\ast \dfrac{t^{-\beta}}{\Gamma(1-\beta)}$, and $0<\rho(t)\leq k t^{-\delta}$ (see \cite{pal7}), $\delta=\beta-\alpha$.

The following estimate will give the continuity at $t=0$ of $({\bf u},{\bf S})$.

\begin{lemma}\label{sol5}
For any $\mu_0>0$, $\exists M>0$,  such that $\forall ({\bf x},t)\in(\mathbb{R}_+)^3 \times \mathbb{R}_+ $, and $\forall \mu\geq\mu_0$, we have:

\beq\label{sol6}
\displaystyle\mathop{\lim}_{A\to\infty}\displaystyle\left|\int_{-A}^{+A}T_\mu w(x+iy)e^{(x+iy)t} \ud y  \right| \leq \displaystyle\int_{-\infty}^{+\infty}\left|(T_\mu w)(iy) \right|\ud y \leq \dfrac{M}{\mu^{1/(2-\delta)}}
\eeq
 
\end{lemma}

\begin{proof}

Whenever $y>0$, $\displaystyle \left|(T_\mu w)(iy) \right|=\dfrac{1}{|y|^\delta}\dfrac{1}{\left|\mu+\left(y e^{i\pi/2} \right)^{2-\delta}+ \left(y e^{i\pi/2} \right)^{2-\beta} \right|} $.  Therefore,

\begin{eqnarray}\label{sol7 5}
\left|\mu+\left(y e^{i\pi/2} \right)^{2-\delta}+ \left(y e^{i\pi/2} \right)^{2-\beta} \right| & \geq & \left|\mathrm{Im} \left[ \mu e^{i\pi(\beta-2)/2} + y^{2-\delta} e^{i\pi\alpha/2}+y^{2-\beta}  \right] \right|  \nonumber\\
& \geq &  \left|-\mu \sin(\pi\beta/2)+y^{2-\delta}\sin(\pi\alpha/2) \right|  \nonumber\\
& \geq & \mu \sin(\pi\beta/2)-y^{2-\delta}\sin(\pi\alpha/2)  \nonumber\\
& \geq &  K\mu, \,\text{for}\,y\leq \mu^{1/(2-\delta)} 
\end{eqnarray}

The constant $K=\sin(\pi\beta/2)-\sin(\pi\alpha/2)>0$ is independent of $\mu$.  Moreover,

\begin{eqnarray}\label{sol8 5}
\left|\mu+\left(y e^{i\pi/2} \right)^{2-\delta}+ \left(y e^{i\pi/2} \right)^{2-\beta} \right| \label{sol8 1} & \geq & \left|\mathrm{Im} \left[ \mu+ \left(y e^{i\pi/2} \right)^{2-\delta} + \left(y e^{i\pi/2} \right)^{2-\beta}  \right] \right| \label{sol8 2} \nonumber\\
& = & \left|y^{2-\delta} \sin(\pi-\pi\delta/2)+y^{2-\beta}\sin(\pi-\pi\beta/2) \right| \nonumber\\
& = & \left|y^{2-\delta} \sin(\pi\delta/2)+y^{2-\beta}\sin(\pi\beta/2) \right| \nonumber\\
& \geq &  Ky^{2-\delta}, \,\text{for}\,y\geq \mu_0^{1/(2-\delta)} 
\end{eqnarray}

From the above estimates we infer that:

\begin{eqnarray}\label{sol9 3}
\int_{0}^{+\infty}|(T_\mu w)(iy)|\ud y & = & \int_0^{\mu^{1/(2-\delta)}}\dfrac{\ud y}{y^\delta\left|\mu+\left(y e^{i\pi/2} \right)^{2-\delta}+ \left(y e^{i\pi/2} \right)^{2-\beta} \right|} \nonumber\\
& + & \int_{\mu^{1/(2-\delta)}}^{+\infty}\dfrac{\ud y}{y^\delta\left|\mu+\left(y e^{i\pi/2} \right)^{2-\delta}+ \left(y e^{i\pi/2} \right)^{2-\beta} \right|} \nonumber\\
& \leq & \int_0^{\mu^{1/(2-\delta)}} \dfrac{K}{\mu y^\delta}\ud y+\int_{ \mu^{1/(2-\delta)}}^{+\infty}\dfrac{K}{y^\delta y^{2-\delta}}\ud y \nonumber\\
& \leq & \dfrac{M}{\mu^{1/(2-\delta)}} 
\end{eqnarray}

A similar estimate can be obtained for $\displaystyle\int_{-\infty}^0|(T_\mu w)(iy)|\ud y $.  Combining these results achieves the proof.

\end{proof}

Denote, $\delta=\beta-\alpha$, $\omega=\delta/(2-\delta)$ and notice that $0<\omega<\delta<1$. From now on we shall sometimes write $\alpha_k ({\bf u}_0,{\bf S}_0 )$ instead of $\alpha_k $; of course $\alpha_k $ is linear w.r.t. initial data $ ({\bf u}_0,{\bf S}_0 )$.  Most of the following estimates are already proved in \cite{pal7}, save for those derived from Lemma \ref{sol5}.

\begin{proposition}\label{sot1}
Let ${\bf u}_0\in H$, ${\bf S}_0\in L^2(\Omega)^9$.  Then exists $\exists M>0$, such that, $\forall t \in \mathbb{R}_+$ and $\forall k \in \mathbb{N}^\ast$,

\begin{enumerate}[(i)]
\item \label{i} $|\alpha_k(t)|^2\leq M \left(|\alpha_k^0|^2+\lambda_k^{-\omega}|b_k|^2\right)$.
\item \label{ii} $\lambda_k|\alpha_k(t)|^2\leq M \left(\dfrac{|\alpha_k^0|^2}{t^{2-\delta}}+\dfrac{|b_k|^2}{t^{2-2\delta}} \right)$.
\item \label{iii} for any $\mu \in [0,1]$ and any $\tau \in [0,1]$,

$$ |\alpha_k(t)|^2\leq M \left(\dfrac{|\alpha_k^0|^2}{\lambda_k^{\mu} t^{\mu(2-\delta)}}+\dfrac{|b_k|^2}{\lambda_k^{\tau+(1-\tau)\omega} t^{2\tau(1-\delta)}} \right)$$

\item \label{iv} for any $\mu \in [0,1]$ and any $\tau \in [0,1]$,

$$\lambda_k |\rho \displaystyle \ast \alpha_k |^2(t) \leq M \left( \dfrac{|\alpha_k^0|^2}{\lambda_k^{\mu-1} t^{\mu(2-\delta)+2(\delta-1)}}+\dfrac{|b_k|^2}{\lambda_k^{-(1-\tau)(1-\omega)} t^{-2(1-\tau)(1-\delta) }} \right)$$
\end{enumerate}

\end{proposition}

\begin{proof}
 
(i)  \\

Since $\alpha_k$ is linear with respect to $({\bf u}_0,{\bf S}_0)$, we have, by eq.(121) in Theorem 8.1 in \cite{pal7} and eq.\eqref{per3 2}: 

\beq\label{sot1 1}
|\alpha_k ({\bf u}_0,{\bf S}_0)|^2\leq 2|\alpha_k({\bf u}_0,0)|^2+2|\alpha_k(0,{\bf S}_0)|^2 \leq M \left( |\alpha_k^0|^2+\left| \int_{-\infty}^{+\infty}(T_{\lambda_k} w)(iy)e^{iyt}\ud y\sqrt{\lambda_k}b_k \right|^2\right)
\eeq

Invoking Lemma \ref{sol5} we get

\beq\label{sot1 4}
|\alpha_k({\bf u}_0,{\bf S}_0)|^2\leq M \left( |\alpha_k^0|^2 +\dfrac{|b_k|^2 \lambda_k}{\lambda_k^{2 (2-\delta)}}\right) 
\eeq

which gives (i).\\

(ii)\\

Estimate (ii) is obtainable right away from eq.(122) in Theorem 8.1 in \cite{pal7}, with $M_T$ instead of $M$.  The proof that $M$ can be chosen independently of $T$ is deferred until Corollary \ref{ic1} in Section \ref{ie}.  Hence we take here $M$ independent of $T$ and proceed further on.\\

(iii)\\

Notice first that (ii) above gives

\beq\label{sot1 5}
\left| \alpha_k ({\bf u}_0,0) \right|^2(t) \leq \dfrac{M|\alpha_k^0|^2}{\lambda_k t^{2-\delta}}
\eeq

and

\beq\label{sot1 6}
\left| \alpha_k (0,{\bf S}_0) \right|^2(t) \leq \dfrac{M| b_k |^2}{\lambda_k t^{2-2\delta}}
\eeq

Next, combining eq.\eqref{sot1 5} and eq.\eqref{sot1 4} with ${\bf S}_0={\bf 0}$ on one hand, and eq.\eqref{sot1 6} and eq.\eqref{sot1 4} with ${\bf u}_0={\bf 0}$ on the other, making further use of eq.\eqref{sot1 1} leads to estimate (iii).\\

(iv)\\

Since $0\leq \rho(t) \leq K/t^\delta$, estimate (iii) gives

\begin{eqnarray*}
\sqrt{\lambda_k} \left| \rho \displaystyle \ast \alpha_k \right|(t) & \leq & \sqrt{\lambda_k} \left( \rho \displaystyle \ast |\alpha_k| \right)(t) \\
& \leq & M \left(\dfrac{\sqrt{\lambda_k} |\alpha_k^0| }{\lambda_k^{\mu/2} t^{\mu(2-\delta)/2+\delta-1} } + \dfrac{\sqrt{\lambda_k} |b_k|}{ \lambda_k^{[\tau+\omega(1-\tau)]/2} t^{\tau(1-\delta) +\delta-1} }  \right),
\end{eqnarray*}

from which (iv) is obtained.

\end{proof}

In order to work on spaces $H_\theta$ and $D_\gamma$, we need to reformulate Proposition \ref{sot1}.   Let $[\quad]_+$ denote the positive part of a real number.

\begin{lemma}\label{sxl3}

Let ${\bf u}_0\in H$, ${\bf S}_0\in L^2(\Omega)^9$, and $0\leq\gamma\leq\theta\leq\gamma+1$.  Then $\exists M>0$, so that  $\forall t\geq0$, 

\begin{enumerate}[(i)]
\item \label{i} 

$$\lambda_k^\theta |\alpha_k(t)|^2\leq M\left( \lambda_k^\theta |\alpha_k^0|^2+\dfrac{|\lambda_k|^\gamma |b_k|^2}{t^{2(1-\delta)[(\theta-\gamma-\omega)/(1-\omega)]_+}} \right) $$

\item \label{ii}

$$\lambda_k^\gamma \left\langle \lambda_k|\rho \displaystyle \ast  \alpha_k(t)|^2 (t)\right\rangle\leq M \left(\lambda_k^\theta|\alpha_k^0|^2 t^{\delta(\theta-\gamma-\omega)/\omega}+\lambda_k^\gamma|b_k|^2 \right)  $$  

\end{enumerate}

\end{lemma}

\begin{proof}

(i) \\

Assume $0\leq\gamma\leq\theta\leq\gamma+\omega$.  Multiplying (i) of Theorem \ref{sot1} by $\lambda_k^\theta$ leads to $\lambda_k^\theta|\alpha_k(t)|^2\leq A\left[ \lambda_k^\theta|\alpha_k^0|^2+\lambda_k^{\theta-\omega}|b_k|^2 \right]  $.  Since $\displaystyle\mathop{\min}_{k\geq1}\lambda_k>0$ and $\theta-\omega\leq\gamma $, one gets $\lambda_k^\theta|\alpha_k(t)|^2\leq M\left[ \lambda_k^\theta|\alpha_k^0|^2+\lambda_k^{\gamma}|b_k|^2 \right]  $. 

Assume now that $0\leq\gamma+\omega\leq\theta\leq \gamma+1$.  Use  part (iii) of Theorem \ref{sot1} with $\mu=0$ and $\tau=\dfrac{\theta-\gamma-\omega}{1-\omega}\in[0,1] $.  A simple calculation leads to $\tau+(1-\tau)\omega=\tau(1-\omega)+\omega=(\theta-\gamma-\omega)+\omega=\theta-\gamma$.  Henceforth:

\begin{equation*}
|\alpha_k(t)|^2 \leq M \left( |\alpha_k^0|^2+ \dfrac{|b_k|^2}{\lambda_k^{\theta-\gamma} t^{2(1-\delta)(\theta-\gamma-\omega)/(1-\omega)}} \right) \\
\end{equation*} 

which leads to (i).\\

(ii) \\

Assume $0\leq\gamma\leq\theta\leq \gamma+1$.  Letting $\mu=\gamma-\theta+1 \in [0,1]$ and $\tau=1$ in part (iv) of Lemma \ref{sot1}, gives $\mu(2-\delta)+2(\delta-1)=(\gamma-\theta)(2-\delta)+\delta=-\delta(\theta-\gamma-\omega)/\omega$.  Hence

$$ \lambda_k |\rho\displaystyle \ast \alpha_k|^2(t) \leq 
M \left( \dfrac{|\alpha_k^0|^2}{\lambda_k^{ \gamma-\theta }t^{-\delta(\theta-\gamma-\omega)/ \omega }} 
+|b_k|^2 \right)$$  

which gives (ii). 

\end{proof}

As a consequence, we have:

\begin{corollary}\label{soc}
Let $0\leq \gamma\leq\theta\leq\gamma+1$, ${\bf u}_0\in H_\theta$, ${\bf S}_0\in D_\gamma$.  Then ${\bf u}\in \mathscr{C}^{0}(\mathbb{R}^\ast_+,H_{\gamma+1})$.  In addition, 

\begin{enumerate}[(a)]
\item \label{a} whenever $0\leq \gamma\leq\theta\leq\gamma+\omega $,   ${\bf u}\in \mathscr{C}^{0}(\mathbb{R}_+,H_\theta)$, and ${\bf S}\in \mathscr{C}^{0}(\mathbb{R}^\ast_+,D_\gamma)$;  moreover, ${\bf u}(0)={\bf u}_0$.
\item  \label{b} whenever $0\leq \gamma+\omega\leq\theta $,   ${\bf u}\in \mathscr{C}^{0}(\mathbb{R}_+,H_{\gamma+\omega})$, and ${\bf S}\in \mathscr{C}^{0}(\mathbb{R}_+,D_\gamma)$; moreover, ${\bf u}(0)={\bf u}_0$ and ${\bf S}(0)={\bf S}_0$.
\end{enumerate}

In both cases, for any $t\geq0$,

\beq\label{soc1}
\|{\bf u}(t)\|_{H_\theta}\leq M \left( \|{\bf u}_0\|_{H_\theta}+\dfrac{\|{\bf S}_0\|_{D_\gamma}}{t^{(1-\delta)[(\theta-\gamma-\omega)/(1-\omega)]_+}} \right) 
\eeq

\beq\label{soc2}
\|{\bf S}(t)\|_{D_\gamma}\leq M \left( t^{\delta(\theta-\gamma-\omega)/(2\omega)}\|{\bf u}_0\|_{H_\theta}+\|{\bf S}_0\|_{D_\gamma}\right) 
\eeq

\end{corollary}

\begin{proof}

We first prove that ${\bf u}\in \mathscr{C}^{0}\left(\mathbb{R}^\ast_+,H_{\gamma+1}\right)$.  

From (d) of Lemma \ref{sxl4} and (ii) in Proposition \ref{sot1}, we reckon that, for any $N\leq M$ and $t\in[T_1,T_2]$, where $0<T_1<T_2$,

\begin{equation}\label{soc3:a2}
\left\| \displaystyle\sum_{k=N}^{M}\alpha_k(t) {\bf w}_q \right\|^2_{H_{\gamma+1}} =  \displaystyle\sum_{k=N}^{M} |\alpha_k(t)|^2\lambda_k^{\gamma+1} 
\leq  M_T \displaystyle\sum_{k=N}^{+\infty} \left[\dfrac{\lambda_k^\gamma|\alpha_k^0|^2}{T_1^{2-\delta}} + \dfrac{\lambda_k^\gamma|b_k |^2}{T_1^{2(1-\delta)}} \right] 
\end{equation}

Since $\gamma\leq\theta$, we have that ${\bf u}_0\in H_\theta\subset H_\gamma$.  Also ${\bf S}_0\in H_\gamma$.  Hence Eq.\eqref{soc3:a2} implies that 

\beq\label{socx3}
\displaystyle\mathop{\sup}_{t\in[T_1,T_2]} \left\| \displaystyle\sum_{k=N}^{M}\alpha_k(t) {\bf w}_q \right\|_{H_{\gamma+1}} \displaystyle\mathop{\longrightarrow}_{N\to+\infty}0
\eeq

Finally, as $\alpha_k\in \mathscr{C}^{0}\left(\mathbb{R}_+\right)$, from \eqref{socx3} above we deduce that ${\bf u}=\left( \displaystyle\sum_{k=N}^{+\infty}   \alpha_k \otimes {\bf w}_q \right)\in \mathscr{C}^{0}\left(\mathbb{R}^\ast_+,H_{\gamma+1}\right)$. 

Next we proceed with the rest of the proof.\\

(a) \\

Let us prove that ${\bf S}\in \mathscr{C}^{0}(\mathbb{R}^\ast_+,D_\gamma)$.  For any $N \leq M$, by (c) of Lemma \ref{sxl4} and by (ii) of Lemma \ref{sxl3}, 

\begin{eqnarray}\label{sx:x}
\left\|\sum_{k=N}^{M}\left(\rho \ast \alpha_k \right)(t)\left(\nabla{\bf w}_k+\nabla^T{\bf w}_k \right) \right\|^2_{D_\gamma} & = & \sum_{k=N}^{M}|\rho\mathop{\ast}_{(t)}\alpha_k|^2(1+\lambda_k^\gamma)\lambda_k \nonumber\\
& \leq & K\sum_{k=N}^{+\infty} \left[ \lambda^\theta_k |\alpha_k^0|^2 t^{\delta(\theta-\gamma-\omega)/\omega}+\lambda^\gamma_k |b_k|^2 \right] 
\end{eqnarray}

Since ${\bf u}_0\in H_\theta$, and ${\bf S}_0\in D_\gamma$, then $\displaystyle \sum_{k=1}^{+\infty}\left(\rho \displaystyle \ast  \alpha_k \right)(t) \left(\nabla{\bf w}_k+\nabla^T{\bf w}_k \right) $ is uniformly convergent w.r.t. $t$ on any compact subset $[T_0,T_1]\subset\mathbb{R}_+^\ast$, in ${D_\gamma}$.  Given that $\alpha_k\in\mathscr{C}^0(\mathbb{R}_+)$, and that $\rho\in L^1_{\text{loc}}(\mathbb{R}_+)$ - and hence $(\rho\displaystyle  \ast \alpha_k )\in \mathscr{C}^0(\mathbb{R}_+)$ -, we conclude that $\displaystyle \sum_{k=1}^{+\infty}\left(\rho \ast \alpha_k \right) \otimes \left(\nabla{\bf w}_k+\nabla^T{\bf w}_k \right) \in \mathscr{C}^0([T_0,T_1], D_\gamma)$.   As $W_0\in \mathscr{C}^0(\mathbb{R}_+)$, one gets that (see \eqref{t14 3})

\beq\label{soc4}
{\bf S}=\displaystyle \sum_{k=1}^{+\infty}\left(\rho \ast \alpha_k \right) \otimes \left(\nabla{\bf w}_k+\nabla^T{\bf w}_k \right)+W_0\otimes{\bf S}_0 \in \mathscr{C}^0(\mathbb{R}_+, D_\gamma)
\eeq

The inequality Eq.\eqref{soc2} results from eq.\eqref{sx:x} by letting $N=1$, $M=+\infty$, and from the fact that $\| W_0(t){\bf S}_0 \|_{D_\gamma} \leq  \|W_0\|_{\infty}\|{\bf S}_0\|_{D_\gamma}$ for any $t\in\mathbb{R}_+$, since $W_0\in L^{\infty}(\mathbb{R}_+)\cap \mathscr{C}^0(\mathbb{R}_+)$.  In a similar way we prove that ${\bf u} \in \mathscr{C}^0(\mathbb{R}_+, H_\theta)$.  

The inequality eq.\eqref{soc1} is a consequence of (i) in Lemma \ref{sxl3} and of the fact that ${\bf u}\in\mathscr{C}^0(\mathbb{R}_+, H_\theta)$.  Finally,  as $H_{\theta}\hookrightarrow H$ and  $\displaystyle\mathop{\lim}_{t\to0}{\bf u}(t)={\bf u}_0$ in $H$ (see Theorem 8.4 in \cite{pal7}),     $\displaystyle\mathop{\lim}_{t\to0}{\bf u}(t)={\bf u}_0$ in $H_{\theta}$.  \\

(b)\\

Whenever $\theta\geq \omega+\gamma$,   ${\bf u}_0\in H_{\gamma+\omega}$ and ${\bf S}_0\in D_{\gamma}$.  The inequalities (i) and (ii) of Lemma \ref{sxl3}, for $\theta'=\omega+\gamma$ and $\gamma'=\gamma$, read

\begin{eqnarray*}
& & \lambda_k^{\theta'}|\alpha_k(t)|^2\leq M\left(\lambda_k^{\theta'}|\alpha_k^0|^2 +\lambda_k^{\gamma'}|b_k|^2\right) \\
& & \lambda_k^{\gamma'} \left( \lambda_k| \rho \ast \alpha_k|^2(t) \right) \leq M \left( \lambda_k^{\theta'}|\alpha_k^0|^2+\lambda_k^{\gamma'}|b_k|^2 \right)
\end{eqnarray*}

The proof of (i) of Lemma \ref{sxl3} entails the uniform convergence (with respect to $t$ on $[0,+\infty[$) of 

$\displaystyle \sum_{k=1}^{+\infty}\left(\rho \displaystyle \ast  \alpha_k \right)(t)\left(\nabla{\bf w}_k+\nabla^T{\bf w}_k \right)$ in $D_\gamma$.  

Therefore ${\bf S}\in \mathscr{C}^0(\mathbb{R}_+, D_\gamma)$.  Moreover, as $W_0(0)=1$, we get (see \eqref{t14 3}) ${\bf S}(0)=\displaystyle \sum_{k=1}^{+\infty}\left(\rho \ast \alpha_k \right)(0)\left(\nabla{\bf w}_k+\nabla^T{\bf w}_k \right)+W_0(0) {\bf S}_0= {\bf S}_0$.  Arguing as in (a) above one proves that: ${\bf u}\in \mathscr{C}^0(\mathbb{R}_+, H_{\gamma+\omega})$ and ${\bf u}(0)={\bf u}_0$. 

\end{proof}

In the case $\gamma\geq0$ and $\theta=\gamma+\omega$ we get the continuity of $({\bf u},{\bf S})$  at $t=0$:

\begin{theorem}[\textbf{Existence Theorem}] \label{sot2}
Let $\gamma\geq0$.  Assume that ${\bf u}_0\in H_{\gamma+\omega}$ and ${\bf S}_0\in D_{\gamma}$.  Then the system of Eqs.\eqref{fv1}-\eqref{fv2} has at least one solution $({\bf u},{\bf S})\in \left[\mathscr{C}^0\left(\mathbb{R}_+,H_{\gamma+\omega} \right)\cap \mathscr{C}^0(\mathbb{R}^\ast_+,V)\cap L^p_{\text{loc}}(\mathbb{R}_+,V) \right] \times  \mathscr{C}^0(\mathbb{R}_+,D_\gamma)  $, $p\in [1,2/(2-\delta)[$.  Moreover, $\exists A>0$, such that   $\forall t\in\mathbb{R}_+$, 

\begin{equation*}
\|{\bf u}(t)\|_{H_{\gamma+\omega}}+\|{\bf S}(t)\|_{D_{\gamma}}\leq A\left(\|{\bf u}_0\|_{H_{\gamma+\omega}}+\|{\bf S}_0\|_{D_{\gamma}} \right)
\end{equation*}

\end{theorem}

The statement (a) of Corollary \ref{soc} says that, for any $0\leq \gamma \leq \theta \leq \gamma+\omega  $, ${\bf u}_0\in H_\theta$, ${\bf S}_0\in D_\gamma$, and ${\bf S}\in \mathscr{C}^0(\mathbb{R}^\ast_+,H_\gamma)$.  This does not ensure continuity at $t=0$.  Nevertheless,  for $0\leq \gamma \leq \theta \leq \gamma+\omega $,  $\theta\geq\omega$ and still holding on the assumptions ${\bf u}_0\in H_\theta$ and ${\bf S}_0\in D_\gamma$, we have ${\bf S}_0\in D_{\theta-\omega}$.  Therefore (see Theorem \ref{sot2})  ${\bf S}\in \mathscr{C}^0(\mathbb{R}_+,D_{\theta-\omega})$, and ${\bf u}\in \mathscr{C}^0(\mathbb{R}_+,D_{\theta})$.

From now on we shall focus on the case $0 \leq \gamma \leq \theta \leq \omega < 1$.  Proceeding as previously we get:

\begin{lemma}\label{sxl5}
Assume that $0\leq \gamma \leq \theta \leq \omega \leq 1$, ${\bf u}_0\in H_\theta$ and ${\bf S}_0\in D_\gamma$.  Then $\exists M>0$, such that for any $t\geq0$, 

\beq\label{sxl51}
\lambda_k^\theta \left| \alpha_k(t) \right|^2 \leq M \left(\lambda_k^\theta \left| \alpha_k^0 \right|^2+\lambda_k^{\theta-\omega }\left| b_k \right|^2  \right) 
\eeq

and 

\beq\label{sxl52}
\lambda_k^{\theta-\omega}(\lambda_k|\rho\displaystyle  \ast \alpha_k|^2(t))\leq M \left[ \lambda_k^{\theta}|\alpha^0_k|^2+\lambda_k^{\theta-\omega}|b_k|^2 \right] ) 
\eeq

\end{lemma}

\begin{proof}

Eq.\eqref{sxl51} follows from part (i) in Proposition \ref{sot1}.  Next, we use part (iv) in Lemma \ref{sot1} with $\mu=1-\omega$ and $\tau=1$.  It gives $\mu-1=-\omega$ and $\mu(2-\delta)+2(\delta-1)=[1-\delta/(2-\delta)](2-\delta)+2\delta-2=0$.  Henceforth,

\begin{equation*}
\lambda_k (|\rho\displaystyle  \ast \alpha_k|^2(t))\leq M \left(\lambda_k^\omega|\alpha^0_k|^2+ |b_k|^2 \right)
\end{equation*}

which ends the proof.

\end{proof}

Hence:

\begin{corollary}\label{soc5}

Assume $0\leq \gamma \leq \theta $, ${\bf u}_0\in H_\theta$ and ${\bf S}_0\in D_\gamma$.

\begin{enumerate}[(a)]
\item \label{a} whenever $\theta\geq\omega$, ${\bf u}\in\mathscr{C}^0(\mathbb{R}_+,H_\theta)$, and ${\bf S}\in\mathscr{C}^0(\mathbb{R}_+,D_{\theta-\omega})$.  Moreover, $\exists A>0$ s.t. $\forall t\geq0$: 

\beq\label{syc51}
\left\|{\bf u}(t)\right\|_{H_\theta}+\left\|{\bf S}(t)\right\|_{D_{\theta-\omega}} \leq A\left(\left\|{\bf u}_0\right\|_{H_\theta}+\left\|{\bf S}_0\right\|_{D_{\theta-\omega}}  \right)
\eeq

\item \label{b} whenever $\theta \leq  \omega $, ${\bf u}\in\mathscr{C}^0(\mathbb{R}_+,H_\theta)$, and ${\bf S}\in\mathscr{C}^0(\mathbb{R}_+,\Delta_{\theta-\omega})$.  Moreover, $\exists A>0$ s.t. $\forall t\geq0$,

\beq\label{syc52}
\left\|{\bf u}(t)\right\|_{H_\theta}+\left\|{\bf S}(t)\right\|_{\Delta_{\theta-\omega}} \leq A\left(\left\|{\bf u}_0\right\|_{H_\theta}+\left\|{\bf S}_0\right\|_{\Delta_{\theta-\omega}}  \right)
\eeq
\end{enumerate}

\end{corollary}

\begin{proof}

(a)\\

The proof is a direct consequence of the discussion preceding Lemma \ref{sxl5}.\\

(b)\\

That  ${\bf u}\in\mathscr{C}^0(\mathbb{R}_+,H_\theta)$ is a consequence of part (a) of Corollary \ref{soc}.  Next, Lemma \ref{sxl5} and part (e) of Lemma \ref{sxl4} imply that, for any $M\leq N$,

\begin{eqnarray}\label{syc53}
\left\| \displaystyle \sum_{k=M}^{N} \left( \rho\displaystyle  \ast \alpha_k \right)(t)\left(\nabla{\bf w}_k+\nabla^T{\bf w}_k \right) \right\|^2_{\Delta_{\theta-\omega}} & = & 2 \displaystyle \sum_{k=M}^{N} \lambda_k^{\theta-\omega} \left( \lambda_k |\rho\displaystyle \mathop{\ast}_{(t)}\alpha_k|^2 \right) \nonumber\\
& \leq & M \displaystyle \sum_{k=M}^{N} \left( \lambda_k^\theta|\alpha^0_k|^2+\lambda_k^{\theta-\omega} |b_k|^2 \right)
\end{eqnarray}

Since ${\bf u}_0\in H_\theta$, ${\bf S}_0\in D_\gamma \hookrightarrow L^2(\Omega)^9 \hookrightarrow \Delta_{\theta-\omega}$, we get by Eq.\eqref{syc53} and Eq.\eqref{t14 3} that

\beq\label{syc54}
\left( {\bf S}-W_0 \otimes {\bf S}_0 \right) \in \mathscr{C}^0(\mathbb{R}_+,\Delta_{\theta-\omega}) 
\eeq

Moreover, $W_0\in \mathscr{C}^0(\mathbb{R}_+)$ and ${\bf S}_0 \in D_\gamma \hookrightarrow \Delta_{\theta-\omega}$.  Therefore, by Eq.\eqref{syc54}, ${\bf S}\in\mathscr{C}^0(\mathbb{R}_+,\Delta_{\theta-\omega})$.  Inequality Eq.\eqref{syc52} follows right away after invoking Eq.\ref{sxl51}, Eq.\eqref{t14 3}, Eq.\eqref{syc53} with $M=1$ and $N=+\infty$, and that $W_0\in L^\infty(\mathbb{R}_+)$ and ${\bf S}_0\in \Delta_{\theta-\omega}$.

\end{proof}

\begin{remark}\label{rksc}
Using part (b) in Corollary \ref{soc5} and by a density  argument one may prove that, for ${\bf u}_0\in H$ and ${\bf S}_0\in \Delta_{-\omega}$, the system of equations \eqref{bvp1} has a weak solution $({\bf u},{\bf S})\in  \mathscr{C}^0(\mathbb{R}_+,H)\times\mathscr{C}^0(\mathbb{R}_+,\Delta_{-\omega})$.  Of course the integrals have to be replaced by inner product functionals.

\end{remark}

\begin{corollary}\label{soc6}
Assume $0\leq \gamma \leq \theta \leq \omega \leq 1$, ${\bf u}_0\in H_\theta$ and ${\bf S}_0\in D_\gamma$.  Then ${\bf u}\in\mathscr{C}^0(\mathbb{R}_+,H_\theta)$, and ${\bf S}\in\mathscr{C}^0(\mathbb{R}_+,H^{-1}(\Omega)^9)$.  Moreover ${\bf u}(0)={\bf u}_0$, ${\bf S}(0)={\bf S}_0$.
\end{corollary}

\begin{proof}
Corollary \ref{soc5} states that ${\bf S}\in\mathscr{C}^0(\mathbb{R}_+,\Delta_{\theta-\omega})$.  Therefore (see Section \ref{so}), the mapping of $\mathbb{R}_+$ into $H^{-1}(\Omega)^9$ defined by $\mathbb{R}_+ \displaystyle\stackrel{{\bf S}}{\rightarrow}\Delta_{\theta-\omega} \displaystyle\stackrel{i}{\hookrightarrow}\Delta_{-1} \displaystyle\stackrel{r}{\rightarrow}H^{-1}(\Omega)^9$, is continuous; the Corollary statement follows right away.

\end{proof}

Hence, to the first existence and uniqueness theorem, we can add the following conclusion: ${\bf S}\in\mathscr{C}^0(\mathbb{R}_+,\Delta_{-\omega})\hookrightarrow \mathscr{C}^0(\mathbb{R}_+,H^{-1}(\Omega)^9)$ and ${\bf S}(0)={\bf S}_0$.

We now give a second existence and uniqueness Theorem in $H_{\gamma+\omega}\times D_\gamma$ spaces.


\begin{theorem}[\textbf{Second Existence and Uniqueness Theorem}]\label{ut2}

Let $\gamma\geq0$.  Assume that ${\bf u}_0 \in H_{\gamma+\omega}$, ${\bf S}_0 \in D_{\gamma }$.  Then the boundary value problem given by the system of equations \eqref{fv1}-\eqref{fv2} has a unique solution 

\beq\label{ut21}
({\bf u},{\bf S}) \in \left[ \mathscr{C}^0 (\mathbb{R}_+,H_{\gamma+\omega}) \cap \mathscr{C}^0_b (\mathbb{R}_+^\ast,V)\cap L^1_{\text{loc}}(\mathbb{R}_+ ,V) \right] \times  \left[ \mathscr{C}^0  (\mathbb{R}_+,D_\gamma) \cap L^\infty(\mathbb{R}_+,D_\gamma) \right]
\eeq

Moreover, there exists $A>0$, independent of ${\bf u}$, such that, for any $t\geq0$,

\beq\label{ut22}
\|{\bf u}(t)\|_{H_{\gamma+\omega}}+\|{\bf S}(t)\|_{D_{\gamma}} \leq A \left( \|{\bf u}_0\|_{H_{\gamma+\omega}}+\|{\bf S}_0\|_{D_{\gamma}} \right)
\eeq

Lastly, ${\bf u}(0)={\bf u}_0$, ${\bf S}(0)={\bf S}_0$. 

\end{theorem}

\begin{proof}
The solution uniqueness is a consequence of the following inclusions: $H_{\gamma+\omega} \hookrightarrow H$, $D_{\gamma} \hookrightarrow L^2(\Omega)^9$, $\left[ \mathscr{C}^0 (\mathbb{R}_+,H_{\gamma+\omega}) \cap \mathscr{C}^0_b (\mathbb{R}_+^\ast,V)\cap L^1_{\text{loc}}(\mathbb{R}_+ ,V) \right] \times  \left[ \mathscr{C}^0  (\mathbb{R}_+,D_\gamma) \cap L^\infty(\mathbb{R}_+,D_\gamma) \right]  \subset \mathscr{F}$, and of Corollary  \ref{uc1}.  

The existence of a solution $({\bf u},{\bf S}) \in \left[ \mathscr{C}^0 (\mathbb{R}_+,H_{\gamma+\omega}) \cap \mathscr{C}^0  (\mathbb{R}_+^\ast,V)\cap L^1_{\text{loc}}(\mathbb{R}_+ ,V) \right] \times  \left[ \mathscr{C}^0  (\mathbb{R}_+,D_\gamma) \right]$ follows from Theorem \ref{sot2}.  In addition, the last estimate in Theorem \ref{sot2} grants that ${\bf S}\in L^\infty(\mathbb{R}_+,D_\gamma)$.  Next, that $\|{\bf u}(t)\|_V\displaystyle \mathop{\longrightarrow}_{t\to+\infty} 0$ was proved in the first existence and uniqueness Theorem presented above.  Based on this fact, we infer that ${\bf u} \in \mathscr{C}^0_b (\mathbb{R}_+^\ast,V)$, which ends the proof of solution existence.

The estimate Eq.\eqref{ut22} is a consequence of Theorem \ref{sot2}.

\end{proof}

\section{The smoothness of solutions.}\label{ie}

The following estimates will be used in proving the smoothness of solutions.  They generalize those previously obtained in \cite{pal7}.

\begin{proposition}\label{ie1}

$\exists M>0$,  s.t. $\forall (\tau,\chi)\in[0,1]^2$, $\forall(x,t)\in\mathbb{R}_+$, $\forall\mu\geq\lambda_1$, one has:

\begin{enumerate}[(a)]
\item \label{a} $\left| \displaystyle\lim_{A\to+\infty}  \int_{-A}^{+A}T_\mu (x+iy)e^{(x+iy)t}\ud y \right|\leq \dfrac{M}{\mu^{\chi/(2-\delta)}t^\chi}$
\item \label{b} $\sqrt{\mu}\left| \displaystyle\lim_{A\to+\infty}  \int_{-A}^{+A}T_\mu w(x+iy)e^{(x+iy)t}\ud y \right|\leq \dfrac{M}{\mu^{[\omega(1-\tau)+\tau]/2}t^{\tau(1-\delta)}}$
\end{enumerate}

\end{proposition}

\begin{proof}

We only have to prove these estimates for $\chi=0$ and $\chi=1$, $\tau=0$ and $\tau=1$. \\ 

(a)\\

The case $\chi=0$ has already been addressed in Lemma 7.4 and Lemma 7.2 in \cite{pal7}.

We now prove the case $\chi=1$. Lemma 7.4  and inequality 70 in \cite{pal7} give, for suitable $\kappa>0$ and $B>0$,

\begin{eqnarray}\label{ie01}
\left| \displaystyle\lim_{A\to+\infty}  \int_{-A}^{+A}T_\mu (x+iy)e^{(x+iy)t}\ud y \right| & \leq & K \left[\displaystyle\int_{0}^{+\infty}\left|T_\mu(ze^{i\pi})-T_\mu(ze^{-i\pi})\right|e^{-zt}\ud z+e^{-\kappa t \mu^{1/(2-\delta)}} \right]  \nonumber \\
& \leq & K \left[ \int_{0}^{+\infty} \dfrac{u^{1-\delta}+\lambda_1^{-(1-\delta)/(2-\delta)}}{(u^{2-\delta}+B)^2}e^{-u t \mu^{1/(2-\delta)}}\ud u + e^{-\kappa t \mu^{1/(2-\delta)}} \right] 
\end{eqnarray}

It implies that:

\begin{equation}\label{ie12}
\left| \displaystyle\lim_{A\to+\infty}  \int_{-A}^{+A}T_\mu (x+iy)e^{(x+iy)t}\ud y \right|  \leq K \left[ \int_{0}^{+\infty} e^{-u t \mu^{1/(2-\delta)}}\ud u + \dfrac{1}{t \mu^{1/(2-\delta)}}\right] \leq \dfrac{A}{t \mu^{1/(2-\delta)}}  
\end{equation}

which gives the statement in (a) for $\chi=1$.  \\

(b)\\

The case $\tau=0$ is addressed in Lemma \ref{sol5}.  The case $\tau=1$: from (iii) in Lemma 7.5 in \cite{pal7} we get

\beq\label{ie13}
\sqrt{\mu}\left| \displaystyle\lim_{A\to+\infty}  \displaystyle\int_{-A}^{+A}T_\mu w (x+iy) e^{(x+iy)t}\ud y \right|\leq \dfrac{K}{\mu t^{1-\delta}}+\dfrac{K e^{-at \mu^{1/(2-\delta)}} }{\mu^{1/2-\delta}}\leq \dfrac{M}{\mu t^{1-\delta}}
\eeq  

as $\displaystyle\mathop{\sup}_{t \mu^{1/(2-\delta)}\geq0}\left|\left(t \mu^{1/(2-\delta)}\right)^{1-\delta}e^{-a t \mu^{1/(2-\delta)}}  \right|\leq M<+\infty$.

\end{proof}

As a consequence we have the following extensions of estimates (iii) and (iv) of Theorem \ref{sot1}.

\begin{corollary}\label{ic1}
Let ${\bf u}_0\in H$, ${\bf S}_0\in L^2(\Omega)^9$.  Then:

\begin{enumerate}[(a)]
\item \label{a} $\exists M >0$, $\forall T\geq0$, , s.t. $\forall \gamma \in [0,\dfrac{1}{2-\delta}]$, $\forall \tau \in [0,1]$, $\forall t \in [0,T]$ and $\forall k\in \mathbb{N}^\ast$,

$$|\alpha_k(t)|^2\leq M  \left(  \dfrac{|\alpha_k^0|^2}{\lambda_k^{2\gamma}t^{2\gamma(2-\delta)}} + \dfrac{|b_k|^2}{\lambda_k^{\tau+\omega(1-\tau)}t^{2\tau(1-\delta)}} \right)$$

\item \label{b} $\exists M >0$, $\forall T\geq0$,  s.t. $\forall \gamma \in [0,1]$, $\forall \tau \in [0,\dfrac{1}{2-\delta}]$, $\forall t \in [0,T]$ and $\forall k\in \mathbb{N}^\ast$,

$$\lambda_k \left| \rho\displaystyle \ast \alpha_k \right|^2(t) \leq M  
\left( \dfrac{|\alpha_k^0|^2}{\lambda_k^{2\gamma-1}t^{2[\gamma(2-\delta)+\delta-1]}} + \dfrac{|b_k|^2}{\lambda_k^{-(1-\tau)(1-\omega)}t^{-2(1-\tau)(1-\delta)}} \right)$$

\end{enumerate}

\end{corollary}

\begin{proof}

(a)\\

The statement in (a) follows from \eqref{per3 2} and Proposition \ref{ie1} (with $\gamma=\dfrac{\chi}{2-\delta}$).  \\

(b) \\

The statement follows from (a) above by convolution.\\

\end{proof}

We deduce from Corollary \ref{ic1}:

\begin{proposition}\label{o1}
 Let $\gamma\geq0$, $\eta>0$.  Then:

\begin{enumerate}[(a)]
\item \label{a} Assume ${\bf u}_0\in H_{1+\gamma+\omega}$ and ${\bf S}_0\in D_{1+\gamma}$.  Then ${\bf u}\in \mathscr{C}^1(\mathbb{R}_+,H_{\gamma})$.
\item \label{b} Assume ${\bf u}_0\in H_{3+\gamma-\omega+\eta}$ and ${\bf S}_0\in D_{3+\gamma }$.  Then ${\bf u}\in \mathscr{C}^2(\mathbb{R}_+^\ast,H_{\gamma})\cap W_{\text{loc}}^{2,p}(\mathbb{R}_+,H_{\gamma})$, $p\in[1,1/(1-\alpha)[$. 
\end{enumerate}

\end{proposition}

\begin{proof}

(a)\\

Based on eq.\eqref{al:a1} and $W_0\in L^{\infty}(\mathbb{R}_+)$, we infer that:

\beq\label{o1 1}
\lambda_k^{\gamma}|\alpha_k'(t)|^2 \leq \lambda_k^{1+\gamma} \left( \lambda_k \left| \rho\displaystyle \ast \alpha_k\right|^2(t)  \right)+\lambda_k^{1+\gamma}|b_k|^2 \|W_0\|_\infty^2
\eeq

Now, Lemma \ref{sxl3} with $\theta=\gamma+\omega$, leads to

\beq\label{o1 2}
\lambda_k^{1+\gamma}\left( \lambda_k \left| \rho\displaystyle \ast \alpha_k\right|^2(t)  \right) \leq A\lambda_k \left( \lambda_k^{\gamma+\omega}|\alpha_k^0|^2+\lambda_k^{\gamma}|b_k|^2 \right) 
\eeq

From Eqs.\eqref{o1 1}-\eqref{o1 2},

\beq\label{o1 3}
\lambda_k^{\gamma}|\alpha_k'(t)|^2 \leq A \left( \lambda_k^{1+\gamma+\omega}|\alpha_k^0|^2+\lambda_k^{1+\gamma}|b_k|^2 \right)
\eeq

Since ${\bf u}_0\in H_{1+\gamma+\omega}$ and ${\bf S}_0\in D_{1+\gamma}$, $\lambda_k^\gamma  |\alpha_k'(t)|^2<+\infty$.  Hence  $\displaystyle\sum_{k=1}^{+\infty} \alpha_k'\otimes {\bf w}_k$ converges in $\mathscr{C}^0(\mathbb{R}_+,H_{ \gamma})$.  Since by Lemma \ref{sxl3}  $\,\displaystyle\sum_{k=1}^{+\infty} \alpha_k\otimes {\bf w}_k$ converges in $\mathscr{C}^0(\mathbb{R}_+,H_{1+ \gamma})$,  it also converges in $ \mathscr{C}^0(\mathbb{R}_+,H_{\gamma})$.  Finally ${\bf u}\in \mathscr{C}^1(\mathbb{R}_+,H_{\gamma})$.  \\

(b)\\

Observe that $1+\gamma+\omega \leq 3+\gamma-\omega+\eta$ and $1+\gamma \leq 3+\gamma $.  Consequently ${\bf u}_0\in H_{1+\gamma+\omega}$, ${\bf S}_0\in H_{1+\gamma}$.  Next, (a) above ensures that ${\bf u}\in \mathscr{C}^1(\mathbb{R}_+,H_{\gamma})$.  We now deduce several estimates for the second order derivatives.  From eq.\eqref{al:a1} it follows that, for $t>0$, $\alpha_k''(t)=-\lambda_k \left( \rho  \displaystyle  \ast \alpha_k' \right)(t) -\sqrt{\lambda_k}b_k W_0'(t)=\lambda_k^2 \left(\rho \displaystyle \ast \rho\displaystyle \ast  \alpha_k \right)(t)+\lambda_k^{3/2}b_k \left(\rho \displaystyle \ast W_0 \right)(t)-\sqrt{\lambda_k}b_k W_0'(t) \in \mathscr{C}^0(\mathbb{R}_+^\ast)$ since $\alpha_k \in  \mathscr{C}^0(\mathbb{R}_+)$, $ W_0 \in \mathscr{C}^0(\mathbb{R}_+) \cap \mathscr{C}^1(\mathbb{R}_+^\ast)$ and $0\leq \rho(t)\leq kt^{-\delta}$.

Let now $\epsilon>0$ be small enough.  Part (a) in Corollary \ref{ic1} with $\gamma=(1-\epsilon)/(2-\delta)\in\left[0,\dfrac{1}{2-\delta}\right]$, $\tau=1 $ and $|W'_0(t)|\leq \dfrac{k}{t^{1-\alpha}}$ (see \cite{pal7}), gives

\begin{eqnarray}\label{xot1}
\lambda_k^{\gamma/2} |\alpha_k''(t)| & \leq & M_T \lambda_k^{2+\gamma/2} \left[ t^{1-2\delta} \displaystyle \ast  \left( \dfrac{|\alpha_k^0|}{\lambda_k^{(1-\epsilon)/(2-\delta)}t^{1-\epsilon}}+\dfrac{|b_k|}{\lambda_k^{1/2 }t^{1-\delta}} \right)  \right] \nonumber\\
& + & M_T \lambda_k^{(1+\gamma)/2}\left( \dfrac{\lambda_k }{t^{\delta-1}}+\dfrac{ 1 }{t^{1-\alpha}} \right)|b_k|  \nonumber\\
& \leq & M_T \left( \dfrac{\lambda_k^{(3+\gamma)/2}}{t^{\delta-1}}+\dfrac{\lambda_k^{(1+\gamma)/2}}{t^{1-\alpha}}  \right) |b_k| \nonumber\\
& + & M_T\dfrac{\lambda_k^{2-(1-\epsilon)/(2-\delta)+\gamma/2}}{t^{2\delta-\epsilon-1}}|\alpha_k^0|
\end{eqnarray}

Observe that, for $\epsilon>0$ small enough, $2(3+\gamma)/2 \leq 3+\gamma $ and $2(\gamma+1)/2 \leq 3+\gamma $.  Also, for $\epsilon>0$ small enough, $2\left[2+\gamma/2-(1-\epsilon)/(2-\delta)  \right] \leq 4+\gamma-2/(2-\delta)+\eta$.

Hence, by \eqref{xot1}, and since ${\bf u}_0 \in H_{3+\gamma-\omega+\eta}=H_{4+\gamma-2/( 2-\delta )+\eta}$ and ${\bf S}_0 \in D_{3+\gamma}$, we get ${\bf u} \in \mathscr{C}^2(\mathbb{R}^\ast_+,H_\gamma)$.

\end{proof}

Proceeding as before (see Proposition \ref{o1}) one obtains the following smoothness properties:

\begin{proposition}\label{ot1}
Let $\gamma\geq0$, $\eta>0$.

\begin{enumerate}[(a)]
\item \label{a} Assume ${\bf u}_0\in H_{2+\gamma-\omega+\eta}$, ${\bf S}_0\in D_{2+\gamma}$.  Then ${\bf S}\in  \mathscr{C}^1(\mathbb{R}_+^\ast,D_\gamma) \cap W_{\text{loc}}^{1,p}(\mathbb{R}_+,D_\gamma)$, for any $p\in [1,1/(1-\alpha)[$.
\item \label{b} Assume ${\bf u}_0\in H_{4+\gamma-\omega+\eta}$, ${\bf S}_0\in D_{4+\gamma}$.  Then ${\bf S}\in  \mathscr{C}^2(\mathbb{R}_+^\ast,D_\gamma) $.
\end{enumerate}

\end{proposition}

\begin{proof}

(a)\\

We limit the proof to the case $0<\eta<\omega$.  Denote $\tilde{\omega}=\omega-\eta$; hence $0<\tilde{\omega}<\omega<1$.

Since ${\bf u}_0\in H_{2+\gamma-\tilde{\omega}}$, ${\bf S}_0\in D_{2+\gamma-\tilde{\omega}-\omega}$, based on Theorem \ref{sot2}, we have that ${\bf S} \in \mathscr{C}^0(\mathbb{R}_+,D_{2+\gamma-\tilde{\omega}-\omega})$.  Next, as $0<\tilde{\omega}<\omega<1$ entails $\gamma<2+\gamma-\tilde{\omega}-\omega$, one gets ${\bf S} \in \mathscr{C}^0(\mathbb{R}_+,D_\gamma)$.  

Next we obtain an estimate for $\displaystyle {\bf S}'(t)$.  Observe first that $W_0\in L^\infty(\mathbb{R}_+)$ and that $\alpha_k'(t)=-\lambda_k \left( \rho\displaystyle \ast \alpha_k \right)(t)-\sqrt{\lambda_k}b_k W_0(t)$.  Therefore:

$$\lambda_k^{\gamma/2}\sqrt{\lambda_k}\left| \rho\displaystyle \ast \alpha_k' \right|\leq A \lambda_k^{(\gamma+1)/2}\left[  \lambda_k \left( \rho\displaystyle \ast \rho\displaystyle \ast |\alpha_k| \right)+\sqrt{\lambda_k} \left(\rho\displaystyle \ast |b_k| \right)  \right](t) $$

Now, part (a) of Corollary \ref{ic1}, with $\tilde{\delta}=2\tilde{\omega}/(1+\tilde{\omega})<2\omega/(1+\omega)=\delta$, $\gamma=1/(2-\tilde{\delta})$, $\tilde{\delta}\in ]0,\delta[$, and $\tau=1$, gives

$$|\alpha_k(t)| \leq M_T \left( \dfrac{|\alpha_k^0|}{\lambda_k^{1/(2-\tilde{\delta})}t^{(2-\delta)/(2-\tilde{\delta})}}+\dfrac{|b_k|}{\sqrt{\lambda_k}t^{1-\delta}}  \right)$$

However, $\left| \rho\displaystyle \ast \rho\right|(t) \leq K t^{1-2\delta}$.  Also, since $0< \tilde{\delta}<\delta$, then $(2-\delta)/(2-\tilde{\delta})\in[0,1[$. One infers that

\begin{eqnarray}\label{ot11}
\lambda_k^{\gamma/2}\sqrt{\lambda_k}\left| \rho\displaystyle \ast \alpha_k' \right|(t) & \leq & M_T \lambda_k^{1+\gamma/2} \bigg[ \sqrt{\lambda_k} \left( \dfrac{1}{t^{2\delta-1}}\displaystyle \ast \dfrac{|\alpha_k^0|}{\lambda_k^{1/(2-\tilde{\delta})}t^{(2-\delta)/(2-\tilde{\delta})}} \right) \nonumber\\
& + &  \sqrt{\lambda_k} \left( \dfrac{1}{t^{2\delta-1}}\displaystyle \ast \dfrac{|b_k|}{\sqrt{\lambda_k}t^{1-\delta}}  \right) + \dfrac{|b_k|}{t^{\delta-1}} \nonumber\\
& \leq & M_T \lambda_k^{1+\gamma/2} \left( \lambda_k^{-\tilde{\omega}/2}t^{2-2\delta-(2-\delta)/(2-\tilde{\delta})}|\alpha_k^0|+t^{1-\delta}|b_k|  \right) 
\end{eqnarray}

Let $a=2-2\delta-(2-\delta)/(2-\tilde{\delta})=1-2\delta+(\delta-\tilde{\delta})/(2-\tilde{\delta}) $, and $b=1-\delta>0$.  Then:

\beq\label{69}
\left| \lambda_k^{\gamma/2}  \sqrt{\lambda_k}\left( \rho\displaystyle \ast \alpha_k' \right) \right|^2 (t) \leq  M_T \left[ \lambda_k^{2+\gamma-\tilde{\omega}}|\alpha_k^0|^2 t^{2a}+\lambda_k^{2+\gamma}t^{2b}|b_k|^2  \right]
\eeq  

Recall that - as stated in (c) of Lemma \ref{sxl4} - that $\left( \nabla {\bf w}_k+ \nabla^T {\bf w}_k\right)_{k\in\mathbb{N}^\ast }$ is an orthogonal sequence of functions that belongs to $D_\gamma$, and $\left\| \nabla {\bf w}_k+ \nabla^T {\bf w}_k \right\|^2_{D_\gamma} = (1+\lambda_k^{\gamma})\lambda_k$.  Consequently, using the estimate given above and that ${\bf u}_0\in H_{2+\gamma-\tilde{\omega}}$, ${\bf S}_0\in D_{2+\gamma}$, leads to the fact that $\displaystyle \sum_{k=1}^{+\infty}\left( \rho \displaystyle  \ast \alpha_k'  \right) \otimes \left( \nabla {\bf w}_k+ \nabla^T {\bf w}_k\right)$ converges in $\mathscr{C}^0(\mathbb{R}_+^\ast,D_\gamma)$.  Next, from Eq in \cite{pal7}, $W_0'\otimes {\bf S}_0\in \mathscr{C}^0(\mathbb{R}_+^\ast,D_\gamma)$.  Therefore $\displaystyle \sum_{k=1}^{+\infty}\left( \rho \displaystyle  \ast \alpha_k'  \right) \otimes \left( \nabla {\bf w}_k+ \nabla^T {\bf w}_k\right)+W_0'\otimes {\bf S}_0\in \mathscr{C}^0(\mathbb{R}_+^\ast,D_\gamma)$.  Hence ${\bf S} \in \mathscr{C}^1(\mathbb{R}_+^\ast,D_\gamma)$.

Whenever $a\geq0$, by Eq.\eqref{69}, $\displaystyle \sum_{k=1}^{+\infty}\left( \rho \displaystyle  \ast \alpha_k'  \right) \otimes \left( \nabla {\bf w}_k+ \nabla^T {\bf w}_k\right)+W_0'\otimes {\bf S}_0$ belongs to $\mathscr{C}^0(\mathbb{R}_+  ,D_\gamma)$, thus belongs to $L^p_{\text{loc}}(\mathbb{R}_+ ,D_\gamma)$ for any $1 \leq p < +\infty$.  

Now, whenever $a<0$, $-a-(1-\alpha)=2\delta-2-(\delta-\tilde{\delta})/(2-\delta)-1+\alpha=[-3-(\delta-\tilde{\delta})/(2-\delta)+(2\delta+\alpha)]<0$.  We conclude that $\displaystyle \sum_{k=1}^{+\infty}\left( \rho \displaystyle  \ast \alpha_k'  \right) \otimes \left( \nabla {\bf w}_k+ \nabla^T {\bf w}_k\right)$ converges in $L^q_{\text{loc}}(\mathbb{R}_+ ,D_\gamma)$, for any  $q\in [1,1/(1-\alpha)]$.  Moreover, from \cite{pal7} we observe that $\| W_0'(t){\bf S}_0 \|_{D_\gamma}\leq K/t^{1-\alpha}$.  It implies that $W_0'\otimes{\bf S}_0\in L^q_{\text{loc}}(\mathbb{R}_+,D_\gamma)$ for any $q\in [1,1/(1-\alpha)[$.  Therefore     

$$\displaystyle \sum_{k=1}^{+\infty}\left( \rho \displaystyle  \ast \alpha_k'  \right) \otimes \left( \nabla {\bf w}_k+ \nabla^T {\bf w}_k\right)+W_0'\otimes{\bf S}_0 \in L^q_{\text{loc}}(\mathbb{R}_+ ,D_\gamma)$$, 

for any $q\in [1,1/(1-\alpha)[$ and irrespective of whether $a$ is positive or negative.

Eventually ${\bf S}\in W_{\text{loc}}^{1,q}(\mathbb{R}_+,D_\gamma)$ for any $q\in [1,1/(1-\alpha)[$.\\

(b)\\

The proof is omitted.\\

\end{proof}

From Proposition \ref{o1} we can infer the existence of smooth solutions to eqs.\eqref{bvp1}.  Assume that ${\bf u}_0 \in H_{5+\omega}$ and ${\bf S}_0 \in D_{5}\cap \mathscr{C}^1(\overline{\Omega})^9$.  Then, the solution $({\bf u},{\bf S})$ the existence of which is granted by Theorem \ref{ut1} of Section \ref{wf}, complies with the statement (a) of Proposition \ref{o1}, that is ${\bf u} \in \mathscr{C}^1(\mathbb{R}_+,H_{4 })$.  Since $H_{4 } \hookrightarrow H^4(\Omega) \hookrightarrow \mathscr{C}^2(\overline{\Omega})$ (see Section \ref{so} and by Sobolev's injection), one has ${\bf u} \in \mathscr{C}^1(\mathbb{R}_+,\mathscr{C}^2(\overline{\Omega})^3 ) $ and $\left( \nabla {\bf u}+ \nabla^T {\bf u} \right) \in \mathscr{C}^1(\mathbb{R}_+,\mathscr{C}^1(\overline{\Omega})^9 ) $.  One also has ${\bf S}_0\in  \mathscr{C}^1(\overline{\Omega})^9 $, $W_0 \in \mathscr{C}^0(\mathbb{R}_+) \cap \mathscr{C}^1(\mathbb{R}_+^\ast)$ and $\rho \in L^1_{\text{loc}}(\mathbb{R}_+)$.  Consequently ${\bf S}=\rho\ast\left( \nabla {\bf u}+ \nabla^T {\bf u} \right)+W_0 \otimes {\bf S}_0 \in \mathscr{C}^0(\mathbb{R}_+ ,\mathscr{C}^1(\overline{\Omega})^9 ) \cap \mathscr{C}^1(\mathbb{R}_+^\ast,\mathscr{C}^1(\overline{\Omega})^9 )$.  All the precedent arguments eventually lead to the conclusion that $({\bf u},{\bf S}) \in \mathscr{C}^1( \mathbb{R}_+ ,\mathscr{C}^2(\overline{\Omega})^3 ) \times \left[ \mathscr{C}^0(\mathbb{R}_+ ,\mathscr{C}^1(\overline{\Omega})^9 ) \cap \mathscr{C}^1(\mathbb{R}_+^\ast ,\mathscr{C}^1(\overline{\Omega})^9 ) \right] $, whenever ${\bf u}_0 \in H_{5+\omega}$ and ${\bf S}_0 \in D_{5 } \cap \mathscr{C}^1(\overline{\Omega})^9$.

\section{Final comments}

Fractional calculus has a long history that parallels the classical analysis \cite{mr,py,skm}.  It has long been used in modeling natural phenomena: for a quick glimpse see for example \cite{ade1,ade2,ade4,ams,agr,bal1,cm,jd1,jd2,dz, gds,han4,k1,ko2,lion,ne1,mak,ok,py,manero2,sered2,tpx,xly,yz2,yz}, and references cited therein.  In particular, fractional derivative CEs have been found to accurately predict stress relaxation of viscoelastic fluids in the glass transition and glassy (high frequency) states.

The results presented here enrich and complement the linear stability analysis within the framework of variational/weak solutions initiated in \cite{pal7}.  We have proved results regarding existence, uniqueness, smoothness and continuity at $t=0$ of the solution to the initial boundary value problem stated in Section \ref{intro}.  Moreover, this work is related to that of Shaw, Whiteman and co-workers on the well posedness, existence and uniqueness of weak solutions for similar in nature hereditary - type integral models (see for example \cite{sw1},\cite{sw2},\cite{sw3},\cite{sw4}), as well as to that reported in \cite{pal5}, \cite{pal6},\cite{vtb}.

The matter of the stability of the original nonlinear CE is an open question on which future work shall focus.

\section{Aknowledgements}

The authors are grateful to Emeritus Associate Professor Michel Charnay, P\^ole de Math\'ematiques, INSA-Lyon, for kind support and encouragements.

\section{Bibliography}

\end{document}